\documentclass{amsart}
\usepackage{amsmath, amssymb}
\usepackage{graphicx}
\usepackage{psfrag}
\usepackage[active]{srcltx}

\newtheorem{thm}{Theorem}[section]

\newtheorem{lem}[thm]{Lemma}
\newtheorem{prop}[thm]{Proposition}
\newtheorem{cond}[thm]{Condition}

\newtheorem{rem}[thm]{Remark}
\numberwithin{equation}{section}
\newtheorem{exam}[thm]{Example}
\newtheorem{assum}[thm]{Assumptions}
\usepackage{amsmath,amsfonts,amssymb,latexsym,epsfig}
\usepackage[colorlinks]{hyperref}

\newcommand{\Hd}{\mathcal{H}}
\newcommand{\E}{\mathbb{E}}

\newcommand{\SJ}{\mathcal{SJD}}
\newcommand{\Vd}{\mathcal{V}}

\newcommand{\psit}{\phi_{{\tiny T}}}

\newcommand{\psith}{\psi^{(\tiny T)}_h}
\newcommand{\tn}{t_n}

\newcommand{\bbR}{\mathbb R}

\newcommand{\mcal}{\mathcal{M}}

\let\phi=\varphi

\newcommand{\R}{{\mathbb R}}

\begin{document}

\title[Optimal tuning of the Hybrid Monte-Carlo Algorithm]{Optimal Tuning of the Hybrid Monte-Carlo Algorithm}%

\author{A. Beskos}
\address{Department of Statistical Science, UCL,
Gower Street, London, WC1E 6BT, UK}%
\email{alex@stats.ucl.ac.uk}%

\author{N.S. Pillai}
\address{CRiSM,
Department of Statistics, University of Warwick, Coventry, CV4 7AL, UK}%
\email{n.pillai@warwick.ac.uk}

\author{G.O.Roberts}
\address{Department of Statistics, University of Warwick, Coventry, CV4 7AL, UK}%
\email{gareth.o.roberts@warwick.ac.uk}%

\author{J. M. Sanz-Serna}
\address{Departamento de Matematica Aplicada,
Facultad de Ciencias, Universidad de Valladolid, Spain }
\email{sanzsern@mac.uva.es}

\author{A.M.Stuart}
\address{Mathematics Institute, University of Warwick, Coventry, CV4 7AL, UK}%
\email{a.m.stuart@warwick.ac.uk}%


\begin{abstract}
We investigate the properties of the Hybrid Monte-Carlo algorithm
(HMC) in high dimensions.  HMC develops a Markov chain reversible
w.r.t.\@ a given target distribution $\Pi$ by using separable 
Hamiltonian dynamics
with potential
$-\log\Pi$. The additional momentum variables are chosen
at random from the Boltzmann distribution and the
continuous-time Hamiltonian dynamics
are then discretised using the leapfrog scheme. The
induced bias is  removed via a Metropolis-Hastings accept/reject
rule. In the simplified scenario of independent, identically
distributed components, we  prove that, to obtain an
$\mathcal{O}(1)$ acceptance probability as the dimension $d$ of
the state space tends to $\infty$, the leapfrog step-size $h$
should be scaled as $h= l \times d^{-1/4}$. Therefore, in high
dimensions, HMC requires $\mathcal{O}(d^{1/4})$ steps to traverse
the state space. We also identify analytically the asymptotically
optimal acceptance probability, which turns out to be  $0.651$
(to three decimal places). This is the choice which optimally
balances the cost of generating a proposal, 
which {\em decreases}
as $l$ increases, against the cost related to the 
average number of proposals required to obtain acceptance,
which {\em increases} as $l$ increases. 
\end{abstract}
\maketitle

\section{Introduction}
The Hybrid Monte Carlo (HMC) algorithm originates from the physics
literature \cite{duna:87} where it was introduced as a fast method
for simulating molecular dynamics. It
has since become  popular  in a number of application areas
including statistical physics \cite{gupt:88,gupt:90,sext:92,hase:01,Akhm:09}, computational chemistry \cite{hans:96,moha:09,Schu:98,tuck:93}, 
data assimilation \cite{alex:etal:05},
geophysics \cite{moha:09} and neural networks 
\cite{neal:96, zloc:01}. The algorithm has also
been proposed as a generic tool for Bayesian statistical 
inference \cite{neal:93,isba:00,giro:09}. 

HMC has been proposed as a method to improve on traditional 
Markov Chain Monte Carlo (MCMC) algorithms. There are heuristic arguments to suggest 
why HMC might perform better, for example based on the idea 
that it breaks down {\em random walk-like} behaviour intrinsic 
to many MCMC algorithms such as Random-Walk Metropolis (RWM) algorithm. 
However there is 
very little theoretical understanding of this phenomenon  (though 
see \cite{Diac:Holm:Neal:00}).
This lack of theoretical guidance of choosing the 
free parameters for the algorithm partly accounts for its relative
obscurity in  statistical applications. The aim of this paper
is to provide insight into 
the behavior of HMC in high dimensions and develop
theoretical tools for improving the efficiency of the algorithm.

HMC uses the derivative of the target probability log-density to
guide the Monte-Carlo trajectory towards areas of high
probability. The standard  RWM algorithm
\cite{metr:53} proposes \emph{local}, symmetric moves around the
current position. In many cases (especially in high dimensions)
the variance of the proposal must be small for the corresponding
acceptance probability to be satisfactory. However smaller
proposal variance leads to higher autocorrelations, and large
computing time to explore the state space. In contrast, and as
discussed in the following sections, HMC exploits the information
on the derivative of the log density
to deliver guided, \emph{global} moves, with
higher acceptance probability.

HMC is closely related to the so-called Metropolis-adjusted
Langevin algorithm (MALA) \cite{robe:96}
which uses the derivative of the log-density  to
propose steepest-ascent moves in the state space. MALA employs
\emph{Langevin} dynamics; the proposal is derived from an Euler
discretisation of a Langevin stochastic differential equation that
leaves the target density invariant. On the other hand, HMC uses
\emph{Hamiltonian} dynamics. The original variable $q$ is seen as
a `location' variable and an auxiliary `momentum' variable $p$ is
introduced; Hamilton's ordinary differential equations are used to
generate moves in the enlarged $(q,p)$ phase space. These moves
preserve the total energy, a fact that implies, in probability
terms, that they preserve the target density $\Pi$ of the original
$q$ variable, provided that the initial momentum is chosen
randomly from an appropriate Gaussian distribution.
Although seemingly of different origin, MALA can be
thought of as a `localised' version of HMC: we will return to this
point in the main text.

In practice, continuous-time Hamiltonian dynamics are  discretised
by means of a numerical scheme; the popular
\emph{St\"ormer-Verlet} or {\em leapfrog} scheme 
\cite{Hair:etal:06,Leim:Reic:04,Sanz:Calv:94,skee:99} 
is  currently the scheme of choice. This
integrator  does not conserve energy exactly and the induced bias
is corrected via a Metropolis-Hastings accept/reject rule. In this
way, HMC develops a Markov chain reversible w.r.t.\@ $\Pi$, whose
transitions incorporate information on $\Pi$ in a natural way.

In this paper we  will investigate the properties of HMC in high
dimensions and, in such a context, offer some guidance over the
\emph{optimal} specification of the free parameters of the
algorithm. We assume that we wish to sample from a density $\Pi$
on $\R^{N}$ with
\begin{equation}
\label{eq:pi}
\Pi(Q) = \exp\bigl(-\Vd(Q)\bigr) \ ,
\end{equation}
for $\Vd:\bbR^{N} \rightarrow \bbR$. We study the simplified
scenario where $\Pi(Q)$ consists of $d\gg 1$ independent
identically distributed (iid) vector  components,
\begin{equation}
\label{eq:iid} \Pi(Q) = \exp\bigl( -\sum_{i=1}^{d} V(q_i)\bigr)\ , 
\quad V: \bbR^m \rightarrow \bbR\ ; \quad N=m\times d \ .
\end{equation}
For the leapfrog integrator, we show analytically that, under
suitable hypotheses on $V$ and as $d\rightarrow \infty$,  HMC
requires $\mathcal{O}(d^{1/4})$ steps to traverse the state 
space, and furthermore, identify the associated optimal 
acceptance probability. 

To be more precise, if $h$ is
the step-size employed in the leapfrog integrator, then
we show that the choice 
\begin{equation}
\textrm{HMC}:\quad h =l\cdot d^{-1/4}
\label{eq:scale1}
\end{equation}
leads to an average acceptance probability 
which is of $\mathcal{O}(1)$
as $d \to \infty:$ Theorem \ref{thm:lim}.
This implies that $\mathcal{O}(d^{1/4})$ steps
are required for HMC to make $\mathcal{O}(1)$ moves in state
space.
Furthermore we provide a result of perhaps greater practical
relevance. We prove that, for the leapfrog integrator
and as $d\rightarrow \infty$, the asymptotically \emph{optimal}
algorithm corresponds to a well-defined value of the acceptance probability,
{\em independent of the particular target} $\Pi$ in
(\ref{eq:iid}). This value is (to three decimal places) 
$0.651$: Theorems \ref{cor:opt} and \ref{thm:sjd}.
Thus, when applying HMC in high
dimensions, one should try to tune the free algorithmic parameters
to obtain an acceptance probability close to that value.
We give the precise definition of optimality when stating
the theorems but, roughly, it is determined by the choice of 
$l$ which 
balances the cost of generating a proposal, 
which {\em decreases}
as $l$ increases, against the cost related to the 
average number of proposals required to obtain acceptance,
which {\em increases} as $l$ increases.

The scaling $\mathcal{O}(d^{1/4})$ to make
$\mathcal{O}(1)$ moves in state space contrasts
favorably with the corresponding scalings $\mathcal{O}(d)$ and
$\mathcal{O}(d^{1/3})$ required in a similar context by RWM and
MALA respectively (see the discussion below). Furthermore,
the full analysis provided in this paper for
the leapfrog scheme may be easily extended to high-order,
volume-preserving, reversible integrators. For such an integrator
the corresponding scaling would be $\mathcal{O}(d^{1/(2\nu)})$,
where $\nu$ (an even integer) represents the order of the method.
For the standard HMC algorithm, previous works have already 
established the relevance of the choice 
$h=\mathcal{O}(d^{-1/4})$ (by heuristic arguments,
see \cite{gupt:90}) and an optimal
acceptance probability of around $0.7$ (by numerical
experiments, see \cite{isba:00}). 
Our analytic study of the scaling issues
in HMC was prompted by these two papers.

The paper is organized as follows. Section \ref{sec:define}
presents the HMC method and reviews the literature concerning
scaling issues  for the  RWM and MALA algorithms.
Section~\ref{sec:limit} studies the asymptotic behaviour of HMC as
the dimensionality grows, $d\rightarrow\infty$, including
the key Theorem \ref{thm:lim}. The optimal tuning
of HMC is discussed in Section \ref{sec:optimal},
including the key Theorems \ref{cor:opt} and
\ref{thm:sjd}. Sections
\ref{sec:energyver} and \ref{sec:proofs} are technical. The 
first of them contains the derivation of the required 
numerical analysis estimates on the leapfrog integrator,
with careful attention paid to the dependence of constants
in error estimates on the initial condition; estimates of 
this kind are not available in the
literature and may be of independent interest. Section
\ref{sec:proofs} gathers the probabilistic proofs. We finish with
some conclusions and discussion in Section \ref{sec:conclusion}.

\section{Hybrid Monte Carlo (HMC)}
\label{sec:define}
\subsection{Hamiltonian dynamics}
Consider the Hamiltonian function:
\begin{equation*}
\Hd(Q,P)=\frac12 \langle P, \mcal^{-1} P\rangle+\Vd(Q) \ ,
\end{equation*}
on $\R^{2N}$, where $\mcal$ is a symmetric positive definite
matrix (the `mass' matrix). One should think of $Q$ as the
\emph{location} argument and $\Vd(Q)$ as the potential energy of
the system; $P$ as the \emph{momenta}, and $(1/2) \langle P,
\mcal^{-1} P\rangle$ as the kinetic energy. Thus $\Hd(Q,P)$ gives
the total \emph{energy}: the sum of the potential and the kinetic
energy.
The Hamiltonian dynamics associated with $\Hd$ are governed by
\begin{equation}
\label{eq:ham1} \frac{dQ}{dt}=\mcal^{-1}P,\quad
\frac{dP}{dt}=-\nabla \mathcal{V}(Q)\ ,
\end{equation}
a system of ordinary differential equations  whose solution flow
$\Phi_t$ defined by
\begin{equation*}
(Q(t), P(t)) = \Phi_t(Q(0), P(0))
\end{equation*}
possesses some key properties relevant to HMC:\\
\begin{itemize}
\item {\bf 1. Conservation of Energy:} The change in the potential becomes kinetic energy;
\textit{i.e.,} $\mathcal{H} \circ \Phi_t = \mathcal{H}$, for all $t>0$, or $\mathcal{H}(\Phi_t(Q(0),P(0))) = \mathcal{H}(Q(0),P(0))$, for all $t>0$ and
all initial conditions $(Q(0),P(0))$.\\
\item{\bf 2. Conservation of Volume:} The volume element $dP\,dQ$ of
the phase space
is conserved under the mapping $\Phi_t$.\\
\item{\bf 3. Time Reversibility:} If $\mathcal{S}$ denotes the symmetry
operator:
\begin{align*}
\mathcal{S}(Q,P) = (Q,-P)
\end{align*}
then $\mathcal{H} \circ \mathcal{S} = \mathcal{H}$ and
\begin{align}
\label{eqn:sym}
\mathcal{S}\circ (\Phi_t)^{-1} \circ \mathcal{S} = \Phi_t \ .
\end{align}
Thus, changing the sign of the initial velocity, evolving
backwards in time, and changing the sign of the final velocity
reproduces the forward evolution.
\end{itemize}

From the Liouville equation for equation \eqref{eq:ham1}
it follows that,
if the initial conditions are distributed according
a probability measure with Lebesgue density depending
only on $\mathcal{H}(Q,P)$, then this probability measure is
preserved by the Hamiltonian flow $\Phi_t.$
In particular, if the initial conditions $(Q(0),P(0))$ of
\eqref{eq:ham1} are distributed with a density
(proportional to) $$\exp ( - \Hd(Q,P))= \exp((1/2) \langle P,
\mcal^{-1} P\rangle)\exp(-\mathcal{V}(Q)),$$ then, 
for all $t>0$, the marginal density of
$Q(t)$ will also be (proportional to) $\exp(-\mathcal{V}(Q))$.
This suggests that integration of equations \eqref{eq:ham1}
might form the basis for an exploration of the
target density $\exp(-\mathcal{V}(Q))$.

\subsection{The HMC algorithm}
To formulate a practical algorithm, the continuous-time dynamics (\ref{eq:ham1}) must
be discretised. 
The most popular \emph{explicit}
method is the St\"{o}rmer-Verlet or leapfrog scheme 
(see \cite{Hair:etal:06,Leim:Reic:04,Sanz:Calv:94} and the references therein)
 defined as follows. Assume a current state $(Q_0,P_0)$; then,
after one step of length  $h>0$ the
system \eqref{eq:ham1} will be at a state $(Q_h,P_h)$ defined by
the three-stage procedure:
\begin{subequations}
\label{leap-frog}
\begin{equation}
P_{h/2}=P_0-\tfrac{h}{2}\,\nabla\Vd(Q_{0})\ ;
\end{equation}
\begin{equation}
Q_{h}=Q_0+h\,\mcal^{-1}P_{h/2}\ ;
\end{equation}
\begin{equation}
P_{h}=P_{h/2}-\tfrac{h}{2}\,\nabla\Vd(Q_{h})\ .
\end{equation}
\end{subequations}
The scheme gives rise to a map:
\begin{equation*}
\Psi_{h}\colon(Q_0,P_0) \mapsto (Q_h,P_h)
\end{equation*}
which approximates the flow $\Phi_h$. The solution at time $T$ is
approximated by taking $\lfloor \frac{T}{h} \rfloor$ leapfrog
steps:
\begin{equation*}
(Q(T), P(T)) = \Phi_T((Q(0), P(0))\approx \Psi^{\lfloor
\frac{T}{h} \rfloor}_h((Q(0), P(0)) \ .
\end{equation*}
Note that this is  a  \emph{deterministic} computation.  The map
\begin{equation*}
\Psi^{(T)}_h := \Psi^{\lfloor \frac{T}{h} \rfloor}_h
\end{equation*}
may be shown to be volume preserving and time reversible
(see \cite{Hair:etal:06,Leim:Reic:04,Sanz:Calv:94}) 
but it does not exactly conserve  energy.
As a consequence the leapfrog algorithm does not share
the property of equations \eqref{eq:ham1} following
from the Liouville equation, namely
that any probability density function proportional to $\exp\bigl(-{\mathcal H}(Q,P)\bigr)$
is preserved. In order to restore this property an
accept-reject step must be added. Paper \cite{neal:93}
provides a clear derivation of the required acceptance
ceriterion.

We can now describe the complete HMC algorithm. Let the current state be~$Q$. The  next state
for the HMC Markov chain is determined by the dynamics described in Table \ref{tb:HMC}.

\begin{table}[!h]
\begin{flushleft}
\medskip
\hrule
\medskip
{\itshape HMC($Q$):}
{\itshape
\vspace{0.2cm}
\begin{enumerate}
\item[(i)] Sample a momentum $P\sim N(0,\mcal)$. \vspace{0.2cm}
\item[(ii)] Accept the proposed update $Q'$ defined via $(Q',P') =
\Psi^{(T)}_h(Q,P)$ w.p.: \vspace{0.2cm}
$$
a((Q,P),(Q',P')):= 1 \wedge \exp\{\Hd(Q,P)-\Hd(Q',P')\} \ .
$$
\end{enumerate} }
\medskip
\hrule
\medskip
\end{flushleft}
\caption{The Markov transition for the Hybrid Monte-Carlo
algorithm. Iterative application for a given starting location
$Q^{0}$, will yield a Markov chain $Q^{0},Q^{1},\ldots$}
\label{tb:HMC}
\end{table}

Due to the time reversibility and volume conservation properties
of the integrator map $\Psi^{(T)}_h$, the  recipe  in Table
\ref{tb:HMC} defines (see \cite{duna:87,neal:93}) a Markov chain
reversible w.r.t\@  $\Pi(Q)$; sampling this chain up to
equilibrium will provide correlated samples $Q^n$ from $\Pi(Q)$.
We note that the momentum $P$ is
merely an auxiliary variable and that the user of the algorithm is
free to choose $h$, $T$ and the mass matrix $\mcal$.
In this paper we concetrate on the optimal choice of
$h$, for high dimensional targets.

\subsection{Connection with other Metropolis-Hastings algorithms}
\label{sec:connect} Earlier research has studied the optimal
tuning of  other Metropolis-Hastings algorithms, namely the
Random-Walk Metropolis (RWM) and the Metropolis-adjusted
Lange\-vin algorithm (MALA). In contrast with HMC, whose proposals
involve a deterministic element, those algorithms use updates that
are purely stochastic. For the target density $\Pi(Q)$ in
(\ref{eq:pi}), RWM is specified through the proposed update
\begin{equation*}
Q' = Q + \sqrt{h}\,Z \ ,
\end{equation*}
with $Z\sim N(0,I)$ (this sample case suffices for our 
exposition, but note that $Z$ may be
allowed to have an arbitrary mean zero distribution), while MALA
is determined through the proposal
\begin{equation*}
Q' = Q + \frac{h}{2}\,\nabla\log\Pi(Q) + \sqrt{h}\,Z\ .
\end{equation*}
The density $\Pi$ is invariant for both algorithms when the
proposals are accepted with probability
\begin{equation*}
a(Q,Q') =  1 \wedge \frac{\Pi(Q')T(Q',Q)}{\Pi(Q)T(Q,Q')} \ ,
\end{equation*}
where
\begin{equation*}
T(x,y)=\mathrm{P}\,[\,Q'\in dy\mid Q=x\,]\,/\,dy
\end{equation*}
is the transition density of the proposed update (note that for
RWM the symmetry of the proposal implies $T(Q,Q') = T(Q',Q)$).

The proposal distribution for MALA corresponds to the Euler
discretization of the stochastic differential equation (SDE)
\begin{equation*}
dQ =  \frac{1}{2}\nabla\log\Pi(Q)\,dt + dW ,
\end{equation*}
for which $\Pi$ is an invariant density (here $W$ denotes a
standard Brownian motion). One can easily check that HMC and MALA
are connected because HMC reduces to MALA when $T\equiv h$,
\textit{i.e.,}\@ when the algorithm makes only a single leapfrog
step at each transition of the chain.

Assume now that RWL and MALA are applied with the scalings
\begin{equation}
\textrm{RWM}:\quad h = l\cdot d^{-1}, \qquad \textrm{MALA}:\quad h
= l\cdot d^{-1/3},
\label{eq:scale2}
\end{equation}
for some constant $l>0$, in the simplified scenario where the
target $\Pi$ has the iid structure (\ref{eq:iid}) with $m=1$. The
papers \cite{robe:97}, \cite{robe:98} prove that, as $d\rightarrow
\infty$ and under regularity conditions on $V$ (the function $V$
must be seven times differentiable\footnote{although this is a technical
requirement which may be relaxed}, with all derivatives having
polynomial growth bounds, and all moments of $\exp(-V)$ must be
finite), the acceptance probability approaches a nontrivial value:
\begin{equation*}
\mathbb{E}\,[\,a(Q,Q')\,]\rightarrow a(l) \in (0,1)
\end{equation*}
(the limit $a(l)$ is different for each of the two algorithms).
Furthermore,  if $q_1^{0},q_1^{1},\ldots$ denotes the projection
of the trajectory $Q^{0},Q^{1},\ldots$ onto its first coordinate,
in the above scenario it is possible to show (\cite{robe:97},
\cite{robe:98}) the convergence of the continuous-time
interpolation
\begin{equation}
\textrm{RWM}:\,\, t\mapsto q_1^{[\,t\cdot d\,]},\qquad \textrm{MALA}:\,\, t\mapsto q_1^{[\,t\cdot d^{1/3}\,]}
\label{eq:scale3}
\end{equation}
to the diffusion process governed by the SDE
\begin{equation}
\label{eq:limsde} dq = -\frac{1}{2}\,l\,a(l)\,V^{'}(q)\,dt +
\sqrt{l\,a(l)}\,dw,
\end{equation}
($w$ represents a standard Brownian motion). 
In view of \eqref{eq:scale2}, \eqref{eq:scale3} and
\eqref{eq:limsde} we deduce that the RWM and MALA
algorithms cost ${\mathcal O}(d)$ and ${\mathcal O}(d^{1/3})$
respectively to explore the invariant measure in
stationarity. Furthermore, as the product $l\,a(l)$
determines the \emph{speed} of the limiting
diffusion the state space  will be explored faster
for the choice $l_{opt}$ of $l$ that maximises $l\,a(l)$. While
$l_{opt}$  depends  on the target distribution, it turns out that
the optimal acceptance probability $a(l_{opt})$ is independent of
$V$. In fact, with three decimal places, one finds:
\begin{equation*}
\textrm{RWM}:\,\, a(l_{opt})=0.234,\qquad \textrm{MALA}:\,\, a(l_{opt})=0.574 \ .
\end{equation*}
Asymptotically as $d\rightarrow\infty$,
this analysis identifies   algorithms that may be
regarded
as \emph{uniformly} optimal, because, as discussed in
\cite{robe:01}, ergodic averages of trajectories corresponding to
$l=l_{opt}$ provide optimal estimation of expectations
$\mathbb{E}\,[\,f(q)\,]$, $q\sim \exp(-V)$, irrespectively of the
choice of the (regular) function $f$. These investigations of the
optimal tuning of RWL and MALA  have been subsequently extended in
\cite{Besk:etal:0901} and \cite{Besk:Stua:09} to non-product
target distributions.

For HMC we show that the scaling \eqref{eq:scale1}
leads to an average acceptance probability of ${\mathcal O}(1)$
and hence to a cost of ${\mathcal O}(d^{1/4})$
to make the ${\mathcal O}(1)$ moves
necessary to explore the invariant measure. However, in
constrast to RWM and MALA, we are not
able to provide a simple description of the limiting
dynamics of a single coordinate of the Markov chain.
Consequently optimality is harder to define.

\section{Hybrid Monte Carlo in the limit $d\rightarrow\infty$.}
\label{sec:limit}

The primary aim of this section is to prove 
Theorem \ref{thm:lim} concerning the scaling of
the step-size $h$ in HMC. We also provide some insight
into the limiting behaviour of the resulting Markov chain,
under this scaling, in Propositions \ref{th:desplaz}
and \ref{thm:lim1}.

\subsection{HMC in the iid scenario}
We now study  the asymptotic behaviour of the HMC
algorithm in the iid scenario \eqref{eq:iid}, when the number $d$
of \lq particles\rq\ goes to infinity. We write
$Q=(q_i)_{i=1}^{d}$ and $P=(p_i)_{i=1}^{d}$ to distinguish the
individual components,
and use the following notation for the combination location/momentum:%
\begin{equation*}
X = (x_i)_{i=1}^d;\quad  x_i := (q_i,p_i) \in \bbR^{2m}.
\end{equation*}
We denote by $\mathcal{P}_q$ and $\mathcal{P}_p$ the projections
onto the position and momentum components of $x$, \emph{i.e.}
$\mathcal{P}_q(q,p) = q$, $\mathcal{P}_p(q,p) = p$.

We have:
\begin{equation*}
\mathcal{H}(Q,P) = \sum_{i=1}^d H(q_i,p_i); \quad
H(q,p) := \frac{1}{2}\langle p, M^{-1} p\rangle+ V(q) \ ,
\end{equation*}
where $M$ is a $m\times m$ symmetric, positive definite matrix.
The Hamiltonian differential equations for a single
($m$-dimensional) particle are then
\begin{equation}\label{eqn:ham2}
\frac{dq}{dt}=M^{-1}p,\quad \frac{dp}{dt}=-\nabla V(q)\ ,
\end{equation}
where ${V}: \mathbb{R}^{m} \rightarrow \bbR$. We denote the
corresponding flow by $\phi_t$ and the leapfrog solution operator
over one $h$-step by $\psi_h$.

Thus the acceptance probability   for the evolution of the $d$
particles is given by (see Table 1):
\begin{equation}
\label{eq:acc}
a(X,Y) = 1 \wedge \exp \Big (
\sum_{i=1}^{d} \bigl[ H (x_i) - H(\psith(x_i)) \bigr] \Big)
\end{equation}
with $Y=(y_i)_{i=1}^{d}=\Psi_h^{(T)}(X)$ denoting the HMC
proposal. Note that  the leapfrog scheme (\ref{leap-frog}) is
applied independently for each of the $d$ particles $(q_i,p_{i})$;
the different co-ordinates are only connected through the
accept/reject decision based on  (\ref{eq:acc}).

\subsection{Energy increments}
Our first aim is to estimate (in an analytical sense) the exponent in the right-hand side
of (\ref{eq:acc}). Since the $d$ particles play the same role, it
is sufficient to study a single term $H (x_i) - H(\psith(x_i))$.
We set
\begin{equation}
\label{eqn:delta} \Delta(x,h) := H(\psith(x) ) - H(\psit(x))   =
H(\psith(x) ) - H(x)  \ .
\end{equation}
This is the energy change, due to the leapfrog scheme, over $0
\leq t \leq T$, with step-size $h$ and initial condition $x$,
which by conservation of energy under the true dynamics,
is simply the energy error at time $T$. We
will  study the first and second moments:
\begin{align}
\mu(h) &:= \E\,[\,\Delta(x,h)\,] =
\int_{\mathbb{R}^{2m}} \Delta(x,h)\, e^{-H(x)}dx \nonumber\ ,\\
s^2(h) &:= \E\,[\Delta(x,h)|^2] \ ,\nonumber
\end{align}
and the corresponding variance
\begin{equation*}
\sigma^2(h) = s^2(h) - \mu^2(h)\ .
\end{equation*}

If
the integrator were exactly energy-preserving, one would have
$\Delta \equiv 0$ and all proposals would be accepted. However it
is well known that the size of $\Delta(x,h)$ is in general no
better than the size of the integration error $ \psith(x)-
\psit(x)$, \emph{i.e.} $\mathcal{O}(h^2)$. In fact, under natural smoothness
assumptions on $V$ the following condition holds (see Section
\ref{sec:energyver} for a proof):
\begin{cond}
\label{cond:delta} There exist functions $\alpha(x)$, $\rho(x,h)$
such that
\begin{align} \label{eqn:deltaerr}
\Delta(x,h) &= h^2 \alpha(x) + h^2\rho(x,h)
\end{align}
with $\lim_{h \rightarrow 0} \rho(x,h) = 0$.
\end{cond}

Furthermore in the proofs of the theorems below we shall use an
additional condition to control the variation of $\Delta$ as a
function of $x$. This condition will be shown in Section
\ref{sec:energyver} to hold under suitable assumptions on the
growth of $V$ and its derivatives.
\begin{cond}
\label{cond:beta}
There exists a function $D: \bbR^{2m} \rightarrow \bbR$ such that
\begin{equation*}
\sup_{0 \leq h \leq 1} \,\frac{|\Delta(x,h)|^2}{h^4} \leq D(x)\ , 
\end{equation*}
with 
\begin{equation*}
\int_{\bbR^{2m}} \,D(x)\, e^{-H(x)} dx < \infty \ .
\end{equation*}
\end{cond}

Key to the proof of  Theorem \ref{thm:lim} is the fact
that the average energy increment scales as ${\mathcal O}(h^4)$.
We show this in Proposition \ref{thm:2musig}
using the following simple lemma that holds for general
volume preserving, time reversible integrators:
\begin{lem} \label{lem:main}
Let $\psith$ be any volume preserving, time reversible numerical
integrator of the Hamiltonian equations (\ref{eqn:ham2}) and
$\Delta(x,h): \bbR^{2m} \times \bbR_+ \rightarrow \bbR$ be  as in
\eqref{eqn:delta}. If $\phi:\R\rightarrow \R$  is an odd function
then:
\begin{equation*}
\int_{\mathbb{R}^{2m}}  \phi(\Delta(x,h))\, e^{-H(x)} \,dx =
-\int_{\mathbb{R}^{2m}}  \phi(\Delta(x,h))\, e^{-H(\psith(x))} \,dx
\end{equation*}
provided at least one of the integrals above exist.
If $\phi$ is an even function, then:
\begin{equation*}
\int_{\mathbb{R}^{2m}}  \phi(\Delta(x,h))\, e^{-H(x)} \,dx =
\int_{\mathbb{R}^{2m}}  \phi(\Delta(x,h))\, e^{-H(\psith(x))} \,dx \ ,
\end{equation*}
provided at least one of the integrals above exist.
\end{lem}
\begin{proof}
See  Section \ref{sec:proofs}.
\end{proof}

Applying this lemma with $\phi(u) = u$, we obtain
\begin{equation*}
\mu(h) =  -\int_{\mathbb{R}^{2m}}  \Delta(x,h)\,e^{-H(\psith (x) )} \,dx \ ,
\end{equation*}
which implies that
\begin{equation}
2\mu(h) =  \int_{\mathbb{R}^{2m}}
\Delta(x,h)\, \big[1 - \exp(-\Delta(x,h))\big]\, e^{-H(x)} \,dx \ .\label{eqn:muhfin}
\end{equation}

We now use first  the  inequality $|e^u - 1| \leq |u| (e^u +1)$
and then Lemma \ref{lem:main}
with $\phi(u) = u^2$ to conclude that
\begin{align}
|2\,\mu(h)| &\leq \int_{\mathbb{R}^{2m}} |\Delta(x,h)|^2 \, e^{-H(\psith(x))} dx +\int_{\mathbb{R}^{2m}} |\Delta(x,h)|^2 \, e^{-H(x)} dx  \nonumber \\
& \leq 2\int_{\mathbb{R}^{2m}} |\Delta(x,h)|^2 e^{-H(x)} dx =  2\,s^2(h)\ .\label{eqn:mubd}
\end{align}
The bound in \eqref{eqn:mubd} is important: it shows  that the
average of $\Delta(x,h)$ is actually of the order of (the average
of) $\Delta(x,h)^2$. Since  for the second-order leapfrog scheme
$\Delta(x,h) = \mathcal{O}(h^2)$,  we see from \eqref{eqn:mubd} that we may
expect the average $\mu(h)$ to actually behave as $\mathcal{O}(h^4)$. This
is made precise in the following theorem.
\begin{prop}
\label{thm:2musig} If the potential $V$ is such that Conditions
\ref{cond:delta} and \ref{cond:beta} hold for the leapfrog
integrator $\psith$, then
\begin{align*}
\lim_{h \rightarrow 0}  \frac{\mu(h)}{h^4} = \mu\ ,\quad
\lim_{h \rightarrow 0}  \frac{\sigma^2(h)}{h^4} = \Sigma \ ,
\end{align*}
for the constants:
\begin{equation*}
\Sigma  = \int_{R^{2m}}\alpha^2(x) \,e^{-H(x)}\,dx\ ;\quad  \mu = \Sigma/2\ .
\end{equation*}
\end{prop}
\begin{proof}
See Section \ref{sec:proofs}.
\end{proof}
Next, we perform explicit calculations for the harmonic oscillator
and verify the conclusions of Proposition \ref{thm:2musig}.
\begin{exam}[Harmonic Oscillator]
\label{exam:harmosc} Consider the Hamiltonian
\begin{equation*}
H(q,p) = \frac12 p^2 + \frac12  q^2 \
\end{equation*}
that gives rise to the system
\begin{equation*}
\left( \begin{matrix}  dq/dt\\ dp/dt \end{matrix} \right)  =
\left( \begin{matrix} p \\ -q \end{matrix} \right) \ ,
\end{equation*}
with solutions
\begin{equation*}
\left( \begin{matrix}  q(t) \\ p(t) \end{matrix} \right)= \left(
\begin{matrix}  \cos( t)&  \sin(t)\\  -  \sin( t) & \cos( t)
\end{matrix} \right) \left( \begin{matrix}  q(0) \\
p(0)\end{matrix} \right) \ .
\end{equation*}

In this case, the leapfrog integration can be written as: 
\begin{equation*}
\psi_h = \psi_h(q,p) = \left( \begin{matrix}   1 - h^2 /2& h
 \\-h + h^3/4 & 1-h^2/2  \end{matrix} \right)\left( \begin{matrix}
q \\ p \end{matrix} \right) \nonumber = \Xi \left( \begin{matrix}
q
\\ p \end{matrix} \right)  \ ,
\label{eqn:matrixeqnexe2}
\end{equation*}
and, accordingly, the numerical solution after $\lfloor
\frac{1}{h} \rfloor$ steps is given by:
\begin{equation*}
\psi_h^{\tiny{(1)}}(q,p)= \Xi^{\lfloor \frac{1}{h} \rfloor}\left(
\begin{matrix}  q \\ p \end{matrix} \right) \ .
\end{equation*}
Diagonalizing $\Xi$ and exponentiating yields:
\begin{equation*}
\Xi^n = \left( \begin{matrix}  \cos(\theta n) &  \frac{1}{ \sqrt{1 -h^2 /4}} \,\sin(\theta n)  \\
-\sqrt{1 -h^2 /4}\,\sin(\theta n)& \cos(\theta n)
\end{matrix} \right)
\end{equation*}
where $\theta = \cos^{-1}(1- h^2/2)$. Using, for instance,
MATHEMATICA, one can now obtain the Taylor expansion:
\begin{equation*}
\Delta(x,h)=  H(\psi_h^{\tiny{(1)}}(x)) - H(x) = h^2  \alpha(x)+
h^4 \beta (x)  +\mathcal{O}(h^6)
\end{equation*}
where:
\begin{gather*}
\alpha(q,p) = \left( (p^2 - q^2 )\sin ^2(1) + p q\sin (2)\right)/8 \ ;\\
\beta(q,p) = \Big(-q^2\sin(2) +pq \big(2 \cos (2)+3 \sin (2)\big)
+p^2 \big(3-3
 \cos (2)+\sin (2)\big)\Big)/192 \ .
\end{gather*}
Notice that, in the stationary regime, $q$, $p$ are standard
normal variables. Therefore, the expectation of $\alpha(x)$ is
$0$. Tedious calculations  give:
\begin{equation*}
\mathrm{Var}\,[\,\alpha(x)\,] =  {1 \over 16} \sin^2(1)\ ,\quad
\E\,[\,\beta(x)\,] = {1 \over 32} \sin^2(1) \ ,
\end{equation*}
in agreement with Proposition \ref{thm:2musig}.
\end{exam}

\subsection{Expected acceptance probability}
We are now in a position to identify the scaling for $h$ that
gives non-trivial acceptance probability as $d\rightarrow\infty$.
\begin{thm}
\label{thm:lim} Assume that the potential $V$ is such that the
leapfrog integrator $\psith$ satisfies Conditions \ref{cond:delta}
and~\ref{cond:beta} and  that
\begin{equation}\label{eq:escala}
h = l\cdot d^{-1/4} \ ,
\end{equation}
for a constant $l>0$. Then in stationarity, \textit{i.e.,} for $X
\sim \exp(-\mathcal{H})$,
\begin{equation*}
\lim_{d\rightarrow\infty}\mathbb{E}\,[\,a(X,Y)\,] =
2\,\Phi(-l^{2}\sqrt{\Sigma}/2)=:a(l)
\end{equation*}
where the constant $\Sigma$ is as defined in Proposition 
\ref{thm:2musig}.
\end{thm}
\begin{proof}
To grasp the main idea, note that the
acceptance probability (\ref{eq:acc}) is given by
\begin{equation}
\label{eq:R} a(X,Y) = 1\wedge e^{R_d}\ ; \qquad R_d =
-\sum_{i=1}^{d} \Delta(x_i,h) \ .
\end{equation}
Due to the simple structure of the target density and
stationarity, the terms $\Delta(x_i,h)$ being added in
\eqref{eq:R} are iid random variables. Since the expectation and
standard deviation of $\Delta(x,h)$ are both $\mathcal{O}(h^4)$  and we have
$d$ terms, the natural scaling to obtain a distributional limit is
given by (\ref{eq:escala}). Then
$R_d \approx N(-\frac12 l^4\Sigma,l^4\Sigma)$
and the desired result follows. See Section \ref{sec:proofs}
for a detailed proof.
\end{proof}
In Theorem \ref{thm:lim} the limit acceptance probability arises
from the use of the Central Limit Theorem. If Condition
\ref{cond:beta} is not satisfied and $\sigma^2(h) = \infty$, then
a Gaussian limit is not guaranteed and it may be necessary to
consider a different scaling to obtain a heavy tailed limiting
distribution such as a stable law.

The  scaling (\ref{eq:escala}) is a direct consequence
of the fact that the leapfrog integrator possesses second order
accuracy. Arguments similar to those  used above prove that the
use of a volume-preserving, symmetric $\nu$-th order integrator
would result in a scaling $h =
\mathcal{O}(d^{-1/(2\nu)})$ ($\nu$ is an even integer)
to obtain an acceptance probability of $\mathcal{O}(1).$

\subsection{The displacement of one particle in a transition}
We now turn our attention to the displacement $q_1^{n+1}-q_1^n$ of
a single particle in a transition $n\rightarrow n+1$ of the chain.
Note that clearly
\begin{equation}
\label{eq:q} q_1^{n+1}= I^{n}\cdot
\mathcal{P}_q\,\psith(q_1^{n},p_1^{n})+\bigl(1-I^{n})q_1^{n};
\quad I^{n} = \mathbb{I}_{\,U^{n}\le a(X^{n},Y^{n})} \ .
\end{equation}

While Conditions \ref{cond:delta} and \ref{cond:beta}
above refer to the error in energy, the
proof of the next results requires a condition on the leapfrog
integration error in the dynamic variables $q$ and $p$. 
In Section \ref{sec:energyver} we describe conditions on $V$
that guarantee the fulfillment of this condition.
\begin{cond}
\label{cond:E} There exists a function $E: \bbR^{2m} \rightarrow
\bbR$ such that
\begin{equation*}
\sup_{0 \leq h \leq 1} \,\frac{| \psith(x) -  \phi_{T}(x)|} {h^2} \leq E(x)\ , 
\end{equation*}
with
\begin{equation*}
\int_{\bbR^{2m}} \,E(x)^4\, e^{-H(x)} dx < \infty \ .
\end{equation*}
\end{cond}

Under the scaling (\ref{eq:escala}) and at stationarity, the
second moment $\E\,[\,(q_1^{n+1}-q_1^{n})^2\,]$ will also approach
a nontrivial limit:

\begin{prop}
\label{th:desplaz} Assume that the hypotheses of Theorem
\ref{thm:lim} and Condition \ref{cond:E} hold and, furthermore, that the
density $\exp(-V(q))$ possesses finite fourth moments. Then,
in stationarity,
\begin{equation*}
\lim_{d\rightarrow\infty}\E\,[\,(q_1^{n+1}-q_1^{n})^2\,]=
C_{J}\cdot a(l)
\end{equation*}
where the value of the constant $C_J$ is given by
\begin{equation*}
C_{J} = \E\,[\,(\mathcal{P}_q \psit(q,p)-q)^2\,]\ ; \quad
(q,p)\sim \exp\bigl(-H(q,p)\bigr) \ .
\end{equation*}
\end{prop}
\begin{proof}
See Section \ref{sec:proofs}.
\end{proof}

We will use this proposition in Section \ref{sec:optimal}.
\subsection{The limit dynamics}
We now discuss the limiting dynamics of the Markov chain,
under the same assumptions made in Proposition \ref{th:desplaz}.
For HCM (as for RWM or MALA) 
the marginal process $\{q_1^n\}_{n\ge 0}$
is not Markovian w.r.t.\@ its own filtration since its dynamics
depend on the current  position of all $d$ particles via the
acceptance probability $a(X^n,Y^n)$ (see (\ref{eq:q})). In the
case of MALA and RWM, $\{q_1^n\}_{n\ge 0}$ is
\emph{asymptotically} Markovian: as $d\rightarrow\infty$ the
effect of the rest of the particles gets averaged to a constant
via the Strong Law of Large Numbers. This allows for the
interpolants of \eqref{eq:scale3}
to converge to solutions of the SDE (\ref{eq:limsde}), which
defines a Markov process. We will now argue that for HCM
$\{q_1^n\}_{n\ge 0}$ cannot be expected to be
\emph{asymptotically} Markovian. In order to simplify the
exposition we will not present all the technicalities of the
argument that follows.

It is well known (see for instance \cite{skee:99})
that, due to time reversibility and under suitable smoothness
assumptions on $V$, the energy increments of the leapfrog
integrator may be expanded in even powers of $h$ as follows (cf.
(\ref{eqn:deltaerr})):
\begin{equation*}
\Delta(x,h) = h^2\alpha(x) + h^4 \beta(x)+ \mathcal{O}(h^6) \ .
\end{equation*}
Here $\E\,[\,\alpha(x)\,]=0$ because from 
Proposition \ref{thm:2musig}
we know that $\E\,[\,\Delta(x,h)\,]=\mathcal{O}(h^4)$.
Ignoring
$\mathcal{O}(h^6)$-terms,  we can write:
\begin{equation*}
a(X^n,Y^n) =  1 \wedge e^{R_{1,d}^n+R_{2,d}^n}\\
\end{equation*}
with
\begin{gather*}
R_{1,d}^n = -h^2\,\sum_{i=1}^{d}\bigl\{
\alpha(x_i^n)- \E\,[\,\alpha(x_{i}^n)\,|\,q_{i}^n\,]\bigr\} -h^4\sum_{i=1}^{d}\beta(x_i^n)\ ,\\
R_{2,d}^n =-
h^2\sum_{i=1}^{d}\E\,[\,\alpha(x_{i}^n)\,|\,q_{i}^n\,]\ \ .
\end{gather*}
Under appropriate conditions, $R_{1,d}^n$ converges, as
$d\rightarrow \infty$, to a Gaussian limit independent of the
$\sigma$-algebra $\sigma(q_{1}^n,q_{2}^n,\ldots)$. To see that,
note that, due to the Strong Law of Large Numbers and since $h^4 =
l^4/d$, the second sum in  $R_{1,d}^n$  converges a.s.\@ to a
constant. Conditionally on $\sigma(q_{1}^n,q_{2}^n,\ldots)$, the
distributional limit of the first term in $R_{1,d}^n$ is Gaussian
with zero mean and a variance determined by the the a.s.\@ limit
of $h^4\sum_{i=1}^{d}\bigl\{ \alpha(x_i^n)-
\E\,[\,\alpha(x_{i}^n)\,|\,q_{i}^n\,]\bigr\}^2$; this follows from
the
Martingale Central Limit Theorem (see e.g.\@ Theorem 3.2 of \cite{hall:80}).
On the other hand, the limit distribution of  $R^n_{2,d}$ is
Gaussian with zero mean but, in general, cannot be asymptotically
independent of $\sigma(q_{1}^0,q_{2}^0,\ldots)$. In the case of
RWM or MALA, the conditional expectations that play the role
played here by $\E\,[\,\alpha(x_{i}^n)\,|\,q_{i}^n\,]$  are
identically zero (see the expansions for the acceptance
probability in \cite{robe:97} and \cite{robe:98}) and this 
implies
that  the corresponding acceptance probabilities are
asymptotically independent from $\sigma(q_{1}^n,q_{2}^n,\ldots)$
and that the marginal processes $\{q_1^n\}_{n\ge 0}$ are
asymptotically Markovian.

The last result in this section provides insight into the limit
dynamics of $\{q_1^n\}_{n\ge 0}$:
\begin{prop}
\label{thm:lim1} Let $Q^{n}\sim \Pi(Q)$,  define
\begin{gather*}
\mathsf{q}_1^{n+1}= \mathsf{I}^{n}\cdot \mathcal{P}_q
\psit(q_1^n,p_1^n)+\bigl(1-\mathsf{I}^{n})q_1^{n}; \quad
\mathsf{I}^{n} = \mathbb{I}_{\,U^{n}\le a(l)} \ ,
\end{gather*}
and consider $q_1^{n+1}$  in (\ref{eq:q}). Then, under the
hypotheses of Proposition \ref{th:desplaz}, as $d\rightarrow \infty$:
\begin{equation*}
(q_1^{n},q_1^{n+1}) \stackrel{\mathcal{L}}{\longrightarrow}
(q_1^{n},\mathsf{q}_1^{n+1})\ .
\end{equation*}
\end{prop}
\begin{proof}
See Section \ref{sec:proofs}.
\end{proof}
\noindent This proposition provides a simple description of the
asymptotic behaviour of the one-transition dynamics of the
marginal trajectories of HMC. As $d\rightarrow\infty$, with
probability $a(l)$, the HMC particle moves under the
\emph{correct} Hamiltonian dynamics. However,
the deviation from the true
Hamiltonian dynamics, due to the energy errors 
accumulated from leapfrog integration of
all $d$ particles, gives rise to the alternative
event of staying at the current position $q^{n}$, 
with probability $1-a(l)$.

\section{Optimal tuning of HMC}
\label{sec:optimal}

In the previous section we addressed
the question of how to scale the step-size in the
leapfrog integration in terms of the dimension $d$, leading
to Theorem \ref{thm:lim}.
In this section we refine this analysis and study the
choice of constant $l$ in \eqref{eq:escala}.
Regardless of the metrics  used to measure the
efficiency of the algorithm, a good choice of $l$ in
(\ref{eq:escala}) has to balance the amount of work needed to
simulate a full $T$-leg (interval of length $T$)
of the Hamiltonian dynamics and the
probability of accepting the resulting proposal. Increasing $l$
decreases the acceptance probability but also decreases
the computational cost of each $T$-leg integration;
decreasing $l$ will yield the
opposite effects, suggesting an optimal value
of $l$.  In this section we present an analysis that
avoids the complex calculations typically associated with
the estimation of mixing times of Markov chains,
but still provides useful guidance regarding the choice of $l$.
We provide two alternative ways of doing this, summarized
in Theorems \ref{cor:opt} and Theorem \ref{thm:sjd}.

\subsection{Asymptotically optimal acceptance probability}
\label{sec:opt}

The number of leapfrog steps of length $h$ needed to compute a
proposal is obviously given by $\lceil T/h\rceil $. Furthermore,
at each step of the chain, it is necessary to evaluate $a(X,Y)$
and sample $P$. Thus the computing time for a single proposal will
be
\begin{equation}
\mathsf{C}_{l,d} := \bigl\lceil \frac{T\,d^{1/4}}{l} \bigr\rceil
\cdot d \cdot  C_{LF} + d\cdot C_{O} \ , \label{eq:cost}
\end{equation}
for some constants $C_{LF}$, $C_{O}$ that measure, for one
particle, the leapfrog costs and the overheads. Let
$\mathsf{E}_{l,d}$ denote the expected computing time until the
first accepted $T$-leg,  in stationarity. If $\mathsf{N}$ denotes
the number of proposals  until (and including) the first to be
accepted, then
\begin{equation*}
\mathsf{E}_{l,d} = \mathsf{C}_{l,d}\,\,\E\,[\,\mathsf{N}\,] =
\mathsf{C}_{l,d}\,\,\E\,[\,\E\,[\,\mathsf{N}\,|\,Q\,]\,] =
\mathsf{C}_{l,d}\,\,
\E\,\bigl[\,\,\frac{1}{\E\,[\,a(X,Y)\,|\,Q\,]}\,\,\bigr] \ .
\end{equation*}
Here we have used the fact that, given the locations $Q$, the
number of proposed $T$-legs follows a geometric distribution with
probability of success $\E\,[\,a(X,Y)\,|\,Q\,]$. Jensen's
inequality yields
\begin{equation}
\label{eq:bound} \mathsf{E}_{l,d} \ge
\frac{\mathsf{C}_{l,d}}{\E\,[\,a(X,Y)\,]} =: \mathsf{E}^{*}_{l,d}\
,
\end{equation}
and, from (\ref{eq:cost}) and  Theorem \ref{thm:lim}, we conclude
that:
\begin{equation*}
\lim_{d\rightarrow \infty} d^{-5/4}\times \mathsf{E}^{*}_{l,d} = \frac{T\, C_{LF}}{a(l)\,l} \ .
\end{equation*}
A sensible choice for $l$ is that which  minimizes the asymptotic
cost $\mathsf{E}^{*}_{l,d}$, that is:
\begin{equation*}
l_{opt}=\arg\max_{l>0} \,\,\textrm{eff}(l)\,;\quad \textrm{eff}(l):=a(l)\,l \ .
\end{equation*}
The value of $l_{opt}$ will in general depend on the specific
target distribution under consideration. However, by expressing
$\textrm{eff}$ as a function of $a=a(l)$, we may write
\begin{equation}
\label{eq:auxiliar} \textrm{eff}=
\bigl(\frac{\sqrt{2}}{\Sigma^{\frac{1}{4}}}\bigr)\cdot  a\cdot
\bigl(\Phi^{-1}\bigl(1-\frac{a}{2}\bigr)\bigr)^{\frac{1}{2}}
\end{equation}
and this equality makes it apparent that  $a(l_{opt})$ {\em does
not vary with the selected target.} 
Fig.\@\ref{fig:opt} illustrates the mapping $a\mapsto
\textrm{eff}(a)$; different choices of target distribution only
change the vertical scale.
In summary, we have:
\begin{thm}
\label{cor:opt} Under the hypotheses of Theorem \ref{thm:lim} and
as $d\rightarrow\infty$, the measure of cost
$\mathsf{E}^{*}_{l,d}$ defined in (\ref{eq:bound}) is minimised
for the choice  $l_{opt}$ of $l$ that leads to the value of
$a=a(l)$ that maximises (\ref{eq:auxiliar}). Rounded to 3 decimal
places the, target independent, optimal value of the limit
probability $a$ is
\begin{equation*}
a(l_{opt}) = 0.651\ .
\end{equation*}
\end{thm}

\begin{figure}
\includegraphics[height=7cm, angle=0]{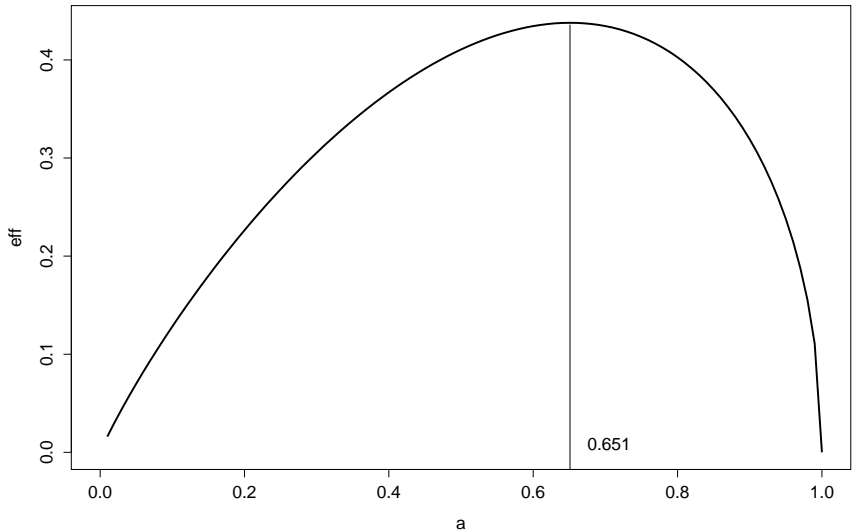}
\caption{The efficiency function $\textrm{eff}=\textrm{eff}(a)$.}
\label{fig:opt}
\end{figure}

The optimal value identified in the preceding theorem is based on
the quantity $\mathsf{E}^{*}_{l,d}$ that underestimates the
expected number of proposals. It may be assumed that the practical
optimal average acceptance probability is in fact {\em greater
than} or equal to 0.651. In the next subsection we use an
alternative measure of efficiency: the expected squared jumping
distance.  Consideration of this alternative metric will also lead
to the same asymptotically optimal acceptance probability of
precisely 0.651 as did the minimisation of $\mathsf{E}^{*}_{l,d}$.
This  suggests that, as $d\rightarrow\infty$, the consequences of
the fact that $\mathsf{E}^{*}_{l,d}$ underestimates
$\mathsf{E}_{l,d}$ become negligible; proving analytically such a
conjecture  seems hard given our current understanding 
of the limiting HMC dynamics.

\subsection{Squared jumping distance}
We now consider the chain $Q^{0},Q^{1},\ldots$ in stationarity
(\emph{i.e.} $Q^{0}\sim \Pi(Q)$) and  account for  the computing
cost $\mathsf{C}_{l,d}$ in (\ref{eq:cost}) by introducing the
continuous-time process $Q^{N(t)}$, where $\{N(t);t\ge 0\}$
denotes a Poisson process of intensity $ \lambda_d  =
1/\mathsf{C}_{l,d} $. If $q_d(t) := q_1^{N(t)}$ denotes the
projection of $Q^{N(t)}$ onto the first particle and $\delta > 0$
is a parameter (the jumping time), we measure the efficiency of
HMC algorithms by using the expected squared jump distance:
\begin{equation*}
\SJ_d(\delta) = \E\,[\,(q_d(t+\delta)-q_d(t))^2\,] \ .
\end{equation*}

The following result shows that $\SJ_d(\delta)$ is indeed
asymptotically maximized by maximizing $a(l)\,l$:
\begin{thm}
\label{thm:sjd} Under the hypotheses of Proposition \ref{th:desplaz}:
\begin{equation*}
\lim_{d\rightarrow\infty} d^{5/4}\times \SJ_{d} =
\frac{C_J\,\delta}{T\,C_{LF}}\times a(l)\,l \ .
\end{equation*}
\end{thm}
\begin{proof}
See Section \ref{sec:proofs}.
\end{proof}
\subsection{Optimal acceptance probability in practice}
\label{sec:practical}
\begin{figure}[!]
\psfrag{abs}[][]{$f(q)=|q|$} \psfrag{xxx}[][]{$f(q)=q^{3}$}
\psfrag{xx}[][]{$f(q)=q^{2}$} \psfrag{x}[][]{$f(q)=q$}
\includegraphics[height=20cm, angle=0]{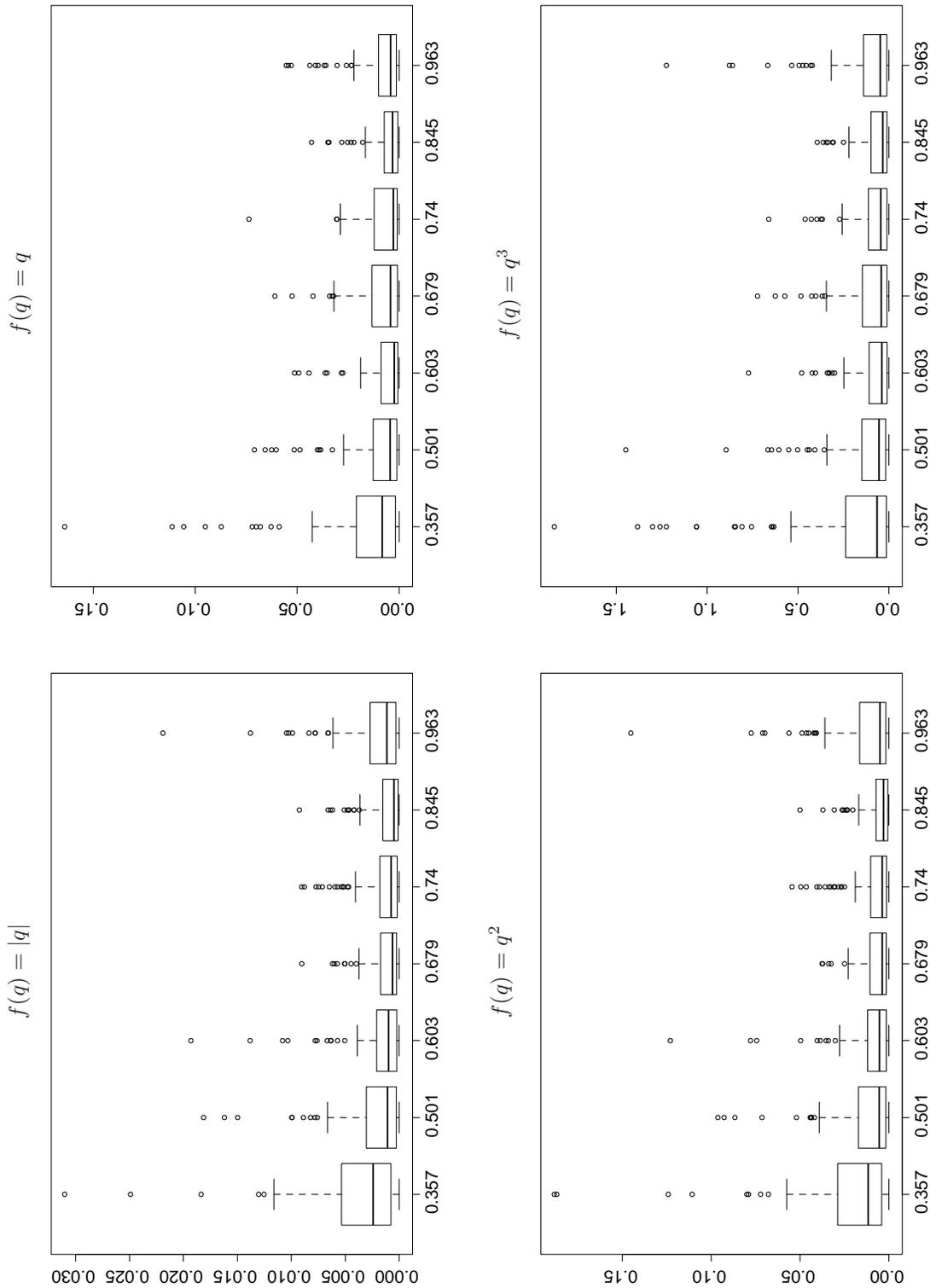}
\caption{Boxplots of Squared Errors (SEs) from Monte-Carlo
averages of HMC. For $7$ different selections of the leapfrog
step-size $h$ (corresponding to the different boxplots in each
panel); the values of $h$ are not shown. We ran HMC $120$ times;
every run was allowed a computing time of 30s. Each boxplot
corresponds to the $120$ SEs in estimating $\E\,[\,f(q)\,]$, for a
particular $h$ and $f(\cdot)$. Written at the bottom of each
boxplots is the median of the $120$ empirical average acceptance
probabilities for the corresponding $h$. } \label{fig:num}
\end{figure}

As $d\rightarrow\infty$, the computing time required for a
proposal scales as $1/l$ (see (\ref{eq:cost})) and the number of
proposals that may be performed in a given amount of time scales
as $l$. Inspection of (\ref{eq:cost}) reveals however that
selecting a big value of $l$ gives the full benefit of a
proportional increase of the number of proposals only
asymptotically, and at the slow rate of $\mathcal{O}(d^{-1/4})$.
On the other hand, the average acceptance probability converges at
the faster rate $\mathcal{O}(d^{-1/2})$ (this is an 
application of Stein's method). These
considerations suggest that unless
$d^{-1/4}$ is  very small the
algorithm will tend to benefit from average acceptance
probabilities higher than $0.651$.

Fig.\@\ref{fig:num} shows the results of a numerical study on HMC.
The target distribution is a product of $d= 10^{5}$ standard
Gaussian densities $N(0,1)$. We have  applied HMC with different
choices of the step-size $h$ and, in all cases,  allowed the
algorithm to run during a computational time $t_{comp}$ of 30
seconds. We used Monte-Carlo averages of the output
\begin{equation*}
\hat{f}= \frac{1}{N_{t_{comp}}}\sum_{n=1}^{N_{t_{comp}}}
f(q_1^{n})
\end{equation*}
to estimate, for different choices of $f$, the expectation
$\E\,[\,f\,]=\E\,[\,f(q)\,]$, $q\sim N(0,1)$; here $N_{t_{comp}}$
denotes the number of $T$-legs carried out within the allowed time
$t_{comp}$. For each choice of $h$ we ran the HMC algorithm $120$
times.

Each of the four panels in Fig.\@\ref{fig:num} corresponds to a
different choice of $f(\cdot)$. In each of the panels, the various
boxplots correspond to  choices of $h$; at the bottom of each
boxplot we have written the median of the $120$ empirical average
acceptance probabilities. The boxplots themselves use the $120$
realizations of the squared distances: $ (\hat{f} -
\E\,[\,f\,])^{2}$. The shape of the boxplots  endorses the point
made above, that the optimal acceptance probability for large (but
finite) $d$ is larger than the asymptotically optimal value of
0.651.
%
%
\section{Estimates for the leapfrog algorithm}
\label{sec:energyver}
In this section we identify  hypotheses on $V$ under which
Conditions \ref{cond:delta}, \ref{cond:beta} and \ref{cond:E}
in Section \ref{sec:limit} hold.

We set $f :=
-\nabla {V}$ (the `force') and denote by $f'(q):= f^{(1)}(q),
f^{(2)}(q),\dots$ the successive Fr\'echet derivatives of $f$ at
$q$. Thus, at a fixed $q$, $f^{(k)}(q)$ is a multilinear operator
from $(\bbR^m)^{k+1}$ to $\bbR$. For the rest of this section we
will use the following assumptions on ${V}$:

\begin{assum} \label{assum:1}
The function $V:\bbR^m \to \bbR$ satisfies:
\begin{itemize}
\item (i) ${V} \in C^4(\bbR^{m}\to \bbR_+).$
\item (ii) $f^{'}, f^{(2)}, f^{(3)}$ are uniformly bounded by a constant $B$.
\end{itemize}
\end{assum}

These assumptions  imply that the potential ${V}(q)$ can grow at
most quadratically at infinity as $|q| \rightarrow \infty$. (If
the growth of $V$ is more than quadratic, then the leapfrog
algorithm as applied with a constant value of $h$ throughout the
phase space is in fact unstable whenever the initial condition is
large.) The case where $V$ takes negative values but is bounded
from below can be reduced to the case $V\geq 0$ by adding a
suitable constant to $V$. In terms of the target measure this just
involves changing the normalization constant and hence is
irrelevant in the HMC algorithm.

\subsection{Preliminaries}
Differentiating \eqref{eqn:ham2} with respect to $t$, we
find successively:
\begin{align*}
\ddot{p}(t) &= f'(q(t)) M^{-1}p(t) \ , \\
\ddot{q}(t) &= M^{-1} f(q(t))\ ,\\
\dddot{p}(t)&= f^{(2)}(q(t))(M^{-1}p(t), M^{-1}p(t)) + f'(q(t))M^{-1} f(q(t)) \ ,\\
\dddot{q}(t) &= M^{-1} f'(q(t))M^{-1}p(t)\ ,\\
\ddddot{p}(t)&= f^{(3)}(q(t))(M^{-1}p(t), M^{-1}p(t), M^{-1}p(t)) + \\
&\quad 3f^{(2)}(q(t))(M^{-1} f(q(t)), M^{-1} p(t))
+ f'(q(t))M^{-1} f'(q(t)) M^{-1} f(q(t))\ ,\\
\ddddot{q}(t) &= M^{-1} f^{(2)}(q(t))(M^{-1}p(t), M^{-1}p(t)) + M^{-1} f'(q(t))M^{-1} f(q(t))\ .
\end{align*}
In this section the letter $K$ will denote a generic constant
which may vary from one appearance to the next, but will depend
only on $B$, $T$, $\|M\|$, $\|M^{-1}\|$. From the above equations
for the derivatives and using the assumptions on $V$, we obtain
the following bounds:

\begin{equation}
\begin{aligned}
|\dot{p}(t)| &\leq |f(q(t))|\ , &|\dot{q}(t)|&\leq K|p(t)|\ , \\
|\ddot{p}(t)| &\leq K|p(t)|\ , &|\ddot{q}(t)| &\leq K|f(q(t))|\ , \\
|\dddot{p}(t)|&\leq K(|p(t)|^2 + |f(q(t))|)\ , &|\dddot{q}(t)| &\leq K|p(t)|\ , \\
|\ddddot{p}(t)|&\leq K(|p(t)|^3 + |p(t)||f(q(t))| + |f(q(t))|)\ , &|\ddddot{q}(t)| &\leq K(|p(t)|^2 + |f(q(t))|) \ .
\end{aligned}
\label{eqn:pqd1}
\end{equation}

\subsection{Asymptotic expansion for the leapfrog solution}
In previous sections we have used a subscript
to denote the different particles comprising our
state space. Here we consider leapfrog integration
of a single particle and use the subscript to denote
the time-level in this integration.
The leapfrog scheme can then be compactly written as
\begin{align}
q_{n+1}  &= q_{n} + h M^{-1}p_n + \frac{h^2}{2} M^{-1} f(q_n)\ , \label{eqn:verlq} \\
p_{n+1} &= p_{n} + \frac{h}{2} f(q_n) + \frac{h}{2} f\Big(q_n + h
M^{-1}p_n + \frac{h^2}{2} M^{-1}f(q_n)\Big)\ .\label{eqn:verlp}
\end{align}

We define the truncation error in the usual way:
\begin{align*}
-\tau^{(q)}_n &:=
q(t_{n+1})- \Big(q(t_n) + hM^{-1}p(t_n) + \frac{h^2}{2} M^{-1} f(q(t_n))\Big)\ ,\\
-\tau^{(p)}_n &:= p(t_{n+1})- \Big(p(t_n) + \frac{h}{2}f(q_n) +
\frac{h}{2} f\big(q(t_n) +hM^{-1}p(t_n) + \frac{h^2}{2}M^{-1}
f(q(t_n)\big)\Big) \ ,
\end{align*}
where we have set $t_n=n h\in[0,T]$.
Expanding (see \cite{Hair:etal:06}) we obtain:
\begin{align*}
\tau^{(q)}_n &= \frac{1}{6}\, h^3\,\dddot{q}(t_n) + h^4\,\mathcal{O}(\|\ddddot{q}(\cdot)\|_{\infty})\ , \\
\tau^{(p)}_n &= -\frac{1}{12}\,h^3 \,\dddot{p}(t_n) +
h^4\,\mathcal{O}(\|\ddddot{p}(\cdot)\|_{\infty}) + h\,\mathcal{O}(\tau^{(q)}_n)\ ,
\end{align*}
where, for arbitrary function $g$: 
\begin{equation*}
\|g(\cdot)\|_{\infty} := \sup_{0 \leq t \leq T} |g(t)|\ .
\end{equation*}
In
view of these estimates, $ (1/6)\,h^3\dddot{q}(t_n)$ and $-(1/12)\,
h^3\,\dddot{p}(t_n)$ are the leading terms in the asymptotic
expansion of the truncation error. Standard results (see, for
instance, \cite{Hair:etal:87}, Section II.8) show that
the numerical solution possesses an asymptotic expansion:
\begin{equation}\label{eqn:asymexh3}
\begin{aligned}
q_{n}  &= q(\tn) + h^2 v(\tn) + \mathcal{O}(h^3)\ ,  \\
p_{n} & = p(\tn)  + h^2 u(\tn) + \mathcal{O}(h^3)\ ,
\end{aligned}
\end{equation}
where functions $u(\cdot)$ and $v(\cdot)$ are the solutions,
with initial condition \mbox{$u(0) = v(0) = 0$}, of the \emph{variational} system
\begin{align}\label{eqn:uvode}
\left( \begin{matrix}  \dot{u}(t) \\ \dot{v}(t) \end{matrix} \right)&= \left( \begin{matrix}   0  & M^{-1}f'(q(t))
\\ I & 0 \end{matrix} \right)\left( \begin{matrix}  u(t) \\ v(t) \end{matrix} \right)
+ \left( \begin{matrix}  \frac{1}{12} \dddot{p}(t) \\ -\frac{1}{6} \dddot{q}(t) \end{matrix}
\right).
\end{align}
\begin{rem}
Notice here that $u(\cdot), v(\cdot)$ depend on the initial conditions
$(q(0), p(0))$ via $(q(\cdot),p(\cdot))$ but this dependence is
not reflected in the notation. One should keep in mind that most of the norms 
appearing in the sequel are functions of $(q(0), p(0))$.
\end{rem}
\noindent Applying Gronwall's lemma and using
the estimates \eqref{eqn:pqd1}, we obtain the bound:
\begin{align} \label{eqn:uvmaxbd}
\|u(\cdot)\|_{\infty} + \|v(\cdot)\|_{\infty} \leq
K  (\|p(\cdot)\|^2_{\infty} +\|f(q(\cdot))\|_{\infty})
\end{align}
and, by differentiating  \eqref{eqn:uvode} with respect to $t$,
expressing $\dot{u}, \dot{v}$ in terms of $u,v$, and using
\eqref{eqn:pqd1} again, we obtain in turn:
\begin{align}
\|\ddot{u}(\cdot)\|_{\infty}  &
\leq K (\|p(\cdot)\|^3_{\infty} +\|p(\cdot)\|_{\infty}\|f(q(\cdot))\|_{\infty}
+\|f(q(\cdot))\|_{\infty})\ ,\label{eqn:nv1} \\
\|\ddot{v}(\cdot)\|_{\infty}  & \leq K (\|p(\cdot)\|^2_{\infty} +
\|f(q(\cdot))\|_{\infty})\ .\label{eqn:nv2}
\end{align}
\subsection{Estimates for the global error}
With the leading coefficients $u$, $v$ of the global errors $q_{n}- q(\tn)$, 
$p_{n} - p(\tn)$ estimated in (\ref{eqn:uvmaxbd}), our
task now is to obtain an explicit bound for the constants implied
in the $\mathcal{O}(h^3)$ remainder in \eqref{eqn:asymexh3}. To this end,
we define the quantities
\begin{align*}
z_n := q(t_n) + h^2 v(t_n)\ ,\\
w_n  := p(t_n) + h^2 u(t_n)\ ,
\end{align*}
and denote by $\tau^{(q)*}_n$, $\tau^{(p)*}_n$ the residuals they generate when substituted in
\eqref{eqn:verlq}, \eqref{eqn:verlp} respectively, \textit{i.e.,}
\begin{align*}
-\tau^{(q)*}_n &= z_{n+1}  - z_n - h M^{-1} w_n - \frac{h^2}{2}M^{-1} f(z_n)\ , \\
-\tau^{(p)*}_n &= w_{n+1} - w_{n} - \frac{h}{2} f(z_n) - \frac{h}{2} f\Big(z_n+ h M^{-1} w_n + \frac{h^2}{2} M^{-1} f(z_n) \Big)\ .
\end{align*}
Since the leapfrog scheme is stable, we have
\begin{align}\label{eqn:maxrecurbd}
\max_{0\leq t_n \leq T}(|q_n - z_n| + |p_n - w_n|) \leq \frac{C}{h} \max_{0\leq t_n \leq T}(|\tau^{(q)*}_n| + |\tau^{(p)*}_n|)
\end{align}
with the constant $C$ depending only on $T$ and Lipschitz constant
of the map $(q_n,p_n) \mapsto (q_{n+1},p_{n+1})$, which in turn
depends on $\|M^{-1}\|$ and the bound for $f'$. The stability
bound (\ref{eqn:maxrecurbd}) is the basis of the proof of the
following estimation of the global error:

\begin{prop} \label{thm:errglobfin}If the potential $V$ satisfies
	Assumptions \ref{assum:1}, then for $0 \leq t_n \leq T$,
\begin{align*}
|p_n - \big(p(t_n) + h^2 u(t_n)\big)| &\leq Kh^3 (\|p(\cdot)\|^4_{\infty}  + \|f(q(\cdot))\|^2_{\infty}+1)\ ,\\
|q_n - \big(q(t_n) + h^2v(t_n)\big)|  &\leq
Kh^3(\|p(\cdot)\|^4_{\infty}  + \|f(q(\cdot))\|^2_{\infty} + 1)\ .
\end{align*}
\end{prop}

\begin{proof}
Our task is reduced to estimating $\tau^{(q)*}_n,\tau^{(p)*}_n$.
We only present the estimation for $\tau^{(p)*}_n$, since the
computations for $\tau^{(q)*}_n$ are similar but simpler.

Indeed, after regrouping the terms,
\begin{align*}
-\tau^{(p)*}_n 
&= \underbrace{p(t_{n+1}) - p(t_n) - \frac{h}{2} f(q(\tn)) - \frac{h}{2} f(q(t_{n+1}))+ \frac{h^3}{12} \ddddot{p}(t)}_{I_1} \\
& + \underbrace{h^2\Big(u(t_{n+1}) - u(t_n) - h f'(q(\tn)) v(t_n)- \frac{h}{12}\dddot{p}(t)\Big)}_{I_2} \\
&+  \underbrace{\frac{h}{2}\Big(f(q(t_n)) - f(z_n) + h^2 f'(q(t_n)) v(t_n)\Big)}_{I_3} \\
&- \underbrace{ \frac{h}{2}\Big(f\big(z_n + h M^{-1}w_n + \frac{h^2}{2} M^{-1} f(q(t_n))\big) - f(q(t_{n+1}) -h^2 f'(q(\tn)) v(t_n)\Big)}_{I_4}\\
&+ \underbrace{\frac{h}{2}\Big(f\big(z_n + h M^{-1}w_n +
\frac{h^2}{2} M^{-1} f(q(t_n)) - f\big(z_n + h M^{-1}w_n +
\frac{h^2}{2} M^{-1} f(z_n)\big)\Big)}_{I_5}
\end{align*}
Now we estimate the above five terms separately.\\
$I_1$: We note that
\begin{align*}
p(t_{n+1}) - p(t_n) - \frac{h}{2} f(q(\tn))& - \frac{h}{2}
f(q(t_{n+1})) = p(t_{n+1}) - p(t_n) - \frac{h}{2}\dot{p}(t_{n+1})
- \frac{h}{2} \dot{p}(t_n)\ .
\end{align*}
and by using the estimates in \eqref{eqn:pqd1} it follows that
\begin{align*}
|I_1| &\leq Kh^4(\|p(\cdot)\|_{\infty} +
\|p(\cdot)\|_{\infty}\|f(q(\cdot))\|_{\infty} +
\|f(q(\cdot))\|_\infty)\ .
\end{align*}
$I_2$: Here we write $I_2  = h^2(u(t_{n+1}) - u(t_n) -
h\,\dot{u}(t_n))$ so that by \eqref{eqn:nv1}
\begin{align*}
|I_2| \leq Kh^4(\|p(\cdot)\|^3_{\infty} +
\|p(\cdot)\|_{\infty}\|f(q(\cdot))\|_{\infty}+\|f(q(\cdot))\|_{\infty}).
\end{align*}
$I_3:$  This term is estimated, after Taylor expanding $f(z_n)$
near $f(q(t_n))$, by
\begin{align*}
|I_3| \leq K h^5(\|p(\cdot)\|_{\infty} +
\|f(q(\cdot))\|_{\infty})^2.
\end{align*}
$I_4:$ We rewrite this as
\begin{align*}
&\frac{h}{2} \Big(f\big(q(t_{n+1}) + \tau^{(q)}_n + h^2 v(t_n) +
h^3 M^{-1} v(t_n)\big) - f(q(t_{n+1})) - h^2 f'(q(t_n))v(t_n)
\Big)
\end{align*}
and Taylor expand around $f(q(t_n))$ to derive the bound:
\begin{align*}
|I_4| \leq Kh^4(\|p(\cdot)\|^4_{\infty}  +
\|f(q(\cdot))\|^2_{\infty}).
\end{align*}
$I_5:$ This term is easily estimated as:
\begin{align*}
|I_5| &\leq Kh^5\|v(\cdot)\|_{\infty} \leq
Kh^5(\|p(\cdot)\|^2_{\infty} + \|f(q(\cdot))\|_{\infty})\ .
\end{align*}

Combining all the above estimates, we have the bound
\begin{align*}
|\tau^{(p)*}_n| \leq Kh^4(\|p(\cdot)\|^4_{\infty}  +
\|f(q(\cdot))\|^2_{\infty})\ .
\end{align*}
A similar analysis for $\tau^{(q)*}_n$   yields the bound
\begin{align*}
|\tau^{(q)*}_n| \leq Kh^4(\|p(\cdot)\|^4_{\infty} + \|f(q(\cdot))\|^2_{\infty})\ .
\end{align*}
The proof  is completed by substituting the above estimates in
\eqref{eqn:maxrecurbd}.
\end{proof}

We now use the
estimates in Proposition \ref{thm:errglobfin}
to derive the asymptotic expansion
for the energy increment for the leapfrog scheme (cf. Condition
1).

\begin{prop} \label{thm:alphabeta} Let  potential $V$ satisfy  Assumptions \ref{assum:1}.
Then, for the leapfrog scheme, we get
\begin{equation*}
\Delta(x,h) =  h^2 \alpha(x)  + h^2\rho(x,h)\ ,
\end{equation*}
with
\begin{align*}
\alpha(x) &= \langle M^{-1}p(T), u(T)\rangle - \langle f(q(T)), v(T) \rangle\ , \\
|\alpha(x)| &\leq K (\|p(\cdot)\|^3_{\infty} + \|f(q(\cdot))\|^2_{\infty} + 1)\ , \\
|\rho(x,h)| &\leq Kh (\|p(\cdot)\|^8_{\infty} +
\|f(q(\cdot))\|^2_{\infty} + 1), \quad 0 < h \leq 1\ ,
\end{align*}
where $(q(\cdot), p(\cdot))$ denotes the solution of \eqref{eqn:ham2} with initial data $x \equiv (q(0),p(0))$
and $u(\cdot), v(\cdot)$ are the solutions of the corresponding variational system given in \eqref{eqn:uvode}
with $u(0) = v(0) =0$.
\end{prop}
\begin{proof}
We only consider
the case when ${T}/{h}$ is an integer. The general case follows with minor
adjustments. By Proposition \ref{thm:errglobfin},
\begin{multline*}
\Delta(x,h) =  H(\psith(x)) - H(x) =H(\psith(x)) - H(\psit(x)) = \\
= \langle M^{-1} p(T), h^2 u(T) + h^3 R_1 \rangle +
\frac{1}{2}\left \langle M^{-1}(h^2 u(T) + h^3 R_1, (h^2 u(T) + h^3 R_1)\right \rangle \\
\qquad + V\Big(q(T) + h^2 v(T) + h^3 R_2\Big) - V(q(T))\ ,
\end{multline*}
where $R_1, R_2$ are remainders with
$$|R_1| + |R_2| \leq
K(\|p(\cdot)\|^4_{\infty} + \|f(q(\cdot))\|^2_{\infty} + 1)\ .$$

By
Taylor expanding $V(\cdot)$ around $q(T)$ we obtain,
\begin{align*}
\Delta(x,h) =  h^2 \big(\langle M^{-1}p(T), u(T)\rangle - \langle
f(q(T)), v(T) \rangle\big) + \rho(x,h)\ ,
\end{align*}
with $$ |\rho(x,h)| \leq Kh^3 (\|p(\cdot)\|^8_{\infty} +
\|f(q(\cdot))\|^2_{\infty} + 1)$$ for $0\leq h \leq 1$. From the
bound \eqref{eqn:uvmaxbd} it follows that
\begin{align*}
|\alpha(x)| &\leq  K(\|p(\cdot)\|_{\infty} \|u(\cdot)\|_{\infty} + \|f(q(\cdot))\|_{\infty} \|v(\cdot)\|_{\infty}) \\
& \leq K(\|p(\cdot)\|^3_{\infty} + \|f(\cdot)\|^2_{\infty} + 1)
\end{align*}
and the theorem is proved.
\end{proof}

Our analysis is completed by estimating the quantities
$\|p(\cdot)\|_{\infty}$ and $\|q(\cdot)\|_{\infty}$, that feature
in the preceding theorems, in terms of the initial data
$(q(0),p(0))$. We obtain these estimates for two families of
potentials which include most of the interesting/useful target
distributions. The corresponding estimates for other potentials
may be obtained using similar methods.
\begin{prop}\label{thm:condvassum}
Let potential $V$ satisfy 
Assumptions \ref{assum:1}. If
$V$ satisfies, in addition, either of the following conditions:
\begin{enumerate}
\item[(i)]  $f$ is bounded and
\begin{align}
\label{eqn:simplevhyp}
\int_{\bbR^{m}} |V(q)|^8 e^{-V(q)} dq < \infty\ ;
\end{align}
\item[(ii)] 
there exist constants $C_1, C_2 >0$ and $ 0< \gamma \leq 1$ such that
for all  $|q| \geq C_2$, we have
$ V(q) \geq C_1|q|^{\gamma}$\,; 
\end{enumerate}
then Conditions \ref{cond:delta}, \ref{cond:beta}
and \ref{cond:E} all  hold.
\end{prop}
\begin{proof}
We only present the treatment of Conditions \ref{cond:delta}
and \ref{cond:beta}. The derivation
of Condition \ref{cond:E} is similar and simpler.

From Proposition \ref{thm:alphabeta} we observe that function
$D(x)$ in Condition \ref{cond:beta} may be taken to be 
\begin{equation*}
D(x) =
K(\|p(\cdot)\|^{16}_{\infty} + \|f(q(\cdot))\|^4_{\infty}+1)\ .
\end{equation*}
Thus,
to prove  integrability of $D(\cdot)$ we need to estimate
$\|p(\cdot)\|_{\infty}$ and $\|f(q(\cdot))\|_{\infty}$.
Estimating $\|p(\cdot)\|_\infty$ is easier. Indeed, by
conservation of energy,
\begin{equation*} 
\frac{1}{2} \langle p(t), M^{-1} p(t)\rangle
\leq  \frac{1}{2} \langle p(0), M^{-1} p(0) \rangle + V(q(0))\ ,
\end{equation*}
which implies 
\begin{equation}
\label{eqn:pinf}
|p(t)|^{16} \leq K(|p(0)|^{16}+ |V(q(0))|^8)\ .
\end{equation}
Now, we prove integrability of $D(\cdot)$ under each of the two stated hypothesis.

Under hypothesis (i): Suppose $f$ is bounded.  In this case we obtain that  $|D(x)| \leq K
(\|p(\cdot)\|^{16}_{\infty} +1 )$, therefore it is enough to
estimate $\|p(\cdot)\|_{\infty}$. Since the Gaussian distribution has
all moments, integrability of $D$ follows from 
\eqref{eqn:simplevhyp}
and \eqref{eqn:pinf}.

Under hypothesis (ii):  Using the stated hypothesis on $V(q)$ we obtain
\begin{equation*}
C_1|q(t)|^{\gamma}  
\leq V(q(t)) \leq \frac{1}{2} \langle p(0), M^{-1} p(0) \rangle + V(q(0))\ ,
\end{equation*}
which implies that:
\begin{equation*}
 |q(t)| \leq K\Big(|p(0)|^{\frac{2}{\gamma}} + |V(q(0))|^{\frac{1}{\gamma}}\Big)\ .
\end{equation*}
By Assumptions \ref{assum:1}(ii), $|f(q(t))| \leq K(1+ |q(t)|)$ and arguing as
above and using the bound \eqref{eqn:pinf}, integrability of $D$
follows if we show that
\begin{align*}
\int_{\bbR^m} |V(q)|^{\delta}\, e^{-V(q)} dq < \infty\ , \quad \delta = \max(8, \frac{4}{\gamma})\ .
\end{align*}
Since $|V(q)| \leq K(1 + |q|^2)$,
\begin{align*}
\int_{\bbR^m} |V(q)|^{\delta}\, e^{-V(q)} dq \leq K\int_{\bbR^m} (1 + |q|^{2\delta})\, e^{-B|q|^\gamma} dq < \infty
\end{align*}
and we are done.
\end{proof}
\section{Proofs of Probabilistic Results}
\label{sec:proofs}

\begin{proof}[Proof of Lemma \ref{lem:main}]
The volume preservation property of $\psith(\cdot)$ implies that
the associated Jacobian is unit. Thus, setting $x =
{(\psith)}^{-1}(y)$ we get:
\begin{align*}
\int_{\mathbb{R}^{2m}}  \phi(\Delta(x,h))\, e^{-H(x)} \,dx &=
\int_{\mathbb{R}^{2m}} \phi\big(H(\psith (x)) - H(x) \big)\, e^{-H(x)} dx  \\
&=\int_{\mathbb{R}^{2m}}  \phi\big[H (y) -
H((\psith)^{-1}(y))\big]\, e^{-H((\psith)^{-1}(y))} dy\ .
\end{align*}
Following the definition of time reversibility in (\ref{eqn:sym}),
we have:
\begin{equation*}
S\circ \psith = (\psith)^{-1}\circ S
\end{equation*}
for the symmetry operator $S$ such that $S(q,p) = (q,-p)$. Using
now the volume presenving transformation $y=S z$  and continuing
from above, we get:
\begin{align*}
&\int_{\mathbb{R}^{2m}}  \phi(\Delta(x,h))\, e^{-H(x)} \,dx\\
&\qquad\qquad=\int_{\mathbb{R}^{2m}}  \phi\big(H (Sz) - H((\psith)^{-1}(Sz) ) \big)\, e^{-H((\psith)^{-1}(Sz))} dz\\
&\qquad\qquad=\int_{\mathbb{R}^{2m}}  \phi\big(H (Sz) - H(S\psith (z) ) \big)\, e^{-H(S (\psith (z)))} dz \\
&\qquad\qquad=\int_{\mathbb{R}^{2m}}  \phi\big(H (z) - H(\psith
(z) ) \big)\, e^{-H( \psith (z))} dz,
\end{align*}
where in the last equation we have used the identity $H(Sz) =
H(z)$.
\end{proof}

\begin{proof}[Proof of Proposition \ref{thm:2musig}]
We will first find the limit of $\sigma^{2}(h)/h^{4}$. Conditions
\ref{cond:delta} and \ref{cond:beta} imply that:
\begin{equation*}
\frac{\Delta^2(x,h)}{h^4} = \alpha^{2}(x) + \rho^{2}(x,h) + 2
\rho(x,h)\alpha(x) \le D(x)
\end{equation*}
and since, for fixed $x$, $\Delta^2(x,h)/h^4\rightarrow
\alpha^{2}(x)$, the dominated convergence theorem shows:
\begin{equation*}
\lim_{h\rightarrow 0}\frac{s^2(h)}{h^4} =
\int_{\R^{2m}}\alpha^2(x)\,e^{-H(x)}dx = \Sigma \ .
\end{equation*}
Now, (\ref{eqn:mubd}) implies that:
\begin{equation}
\label{eq:a1}
\lim_{h\rightarrow 0}\frac{\mu^2(h)}{h^4} = 0 \ ,
\end{equation}
and the required limit for $\sigma^{2}(h)/h^{4}$ follows directly.
Then, from  (\ref{eqn:muhfin}) we obtain
\begin{align*}
\qquad \frac{ 2\mu(h) - \sigma^2(h)}{h^4} & = \nonumber\\ &
\hspace{-1.2cm}-\int_{\mathbb{R}^{2m}}  \frac{\Delta(x,h)}{h^2}
\frac{ \big[\exp(-\Delta(x,h))  - 1 +\Delta(x,h)\big]}{h^2}\,
e^{-H(x)} \,dx + \frac{\mu^2(h)}{h^4} \ .
\end{align*}
Since for any fixed $x$, Conditions \ref{cond:delta} and
\ref{cond:beta} imply that $\Delta(x,h) \rightarrow 0$ as $h
\rightarrow 0$ and $\Delta^2(x,h) = \mathcal{O}(h^4)$, we have the pointwise
limit
\begin{equation*}
\lim_{h \rightarrow 0}\frac{\exp(-\Delta(x,h))  - 1
+\Delta(x,h)}{h^2} = 0\ .
\end{equation*}
Using the inequality $|u||e^u -1- u| \leq |u|^2(e^u + 2)$,
we deduce that for all
sufficiently small $h$,  
\begin{align*}
&\int_{\mathbb{R}^{2m}}  \frac{|\Delta(x,h)|}{h^2}
\frac{ \big|\exp(-\Delta(x,h))\big|}{h^2}\,
e^{-H(x)} \,dx    \\
&\leq \int_{\mathbb{R}^{2m}}  \frac{|\Delta^2(x,h)|}{h^4} \,
{ \exp(-\Delta(x,h)) }\,
e^{-H(x)} \,dx + 2\,\int_{\mathbb{R}^{2m}}  \frac{|\Delta^2(x,h)|}{h^4}\,e^{-H(x)} \,dx\\
&   \leq  3\int_{\bbR^{2m}} D(x)\, e^{-H(x)} dx < \infty,
\end{align*}
%
where the last line follows from applying Lemma \ref{lem:main}
with $\phi(x) = x^2$ and Condition \ref{cond:beta}.
So, the dominated convergence theorem  yields
%
\begin{equation*}
\lim_{h\rightarrow 0}\frac{ 2\mu(h) - \sigma^2(h)}{h^4} = 0 \ .
\end{equation*}
This completes the proof of the proposition.
\end{proof}

\begin{proof}[Proof of Theorem  \ref{thm:lim}]
We continue from (\ref{eq:R}). In view of the scaling  $h=l\cdot
d^{-1/4}$ we obtain, after using  Proposition \ref{thm:2musig}:
\begin{equation*}
\E\,[\,R_d\,] = - d\cdot \mu(h) \rightarrow - \frac{l^4\,\sigma}{2}
\end{equation*}
and
\begin{equation*}
\mathrm{Var}\,[\,R_d\,] = d\cdot \sigma^2(h) \rightarrow
l^4\,\Sigma \ .
\end{equation*}
The Lindeberg condition is easily seen to hold and therefore:
\begin{equation*}
R_d \stackrel{\mathcal{L}}{\longrightarrow} R_{\infty}:= N(-
\tfrac{l^4\,\Sigma}{2},l^4\,\Sigma)  \ .
\end{equation*}
From the boundedness of $u\mapsto 1\wedge e^{u}$ we may write:
\begin{equation*}
\E\,[\,a(X,Y)\,] \rightarrow \E\,[\,1\wedge e^{R_{\infty}}\,]\ ,
\end{equation*}
where the last expectation can be  found analytically (see e.g.\@
\cite{robe:97}) to be:
\begin{equation*}
\E\,[\,1\wedge e^{R_{\infty}}\,]  = 2 \Phi(- l^{2}
\sqrt{\Sigma}/{2}) \ .
\end{equation*}
This  completes the proof.
\end{proof}

\begin{proof}[Proof of Proposition \ref{th:desplaz}]
For simplicity, we will write just $q^{n}$, $q^{n+1}$ and $p^{n}$
instead of $q_{1}^n$, $q_{1}^{n+1}$, $p_{1}^n$ respectively. Using
(\ref{eq:q}), we get:
\begin{equation*}
(q^{n+1}-q^{n})^2 = I^{n}\,( \mathcal{P}_q \psith(q^n,p^n) - q^{n}
)^2 \ .
\end{equation*}
We define:
\begin{equation}
\label{eq:aaa} a^{-}(X^n,Y^n) := 1 \wedge \exp\bigl\{
-\sum_{i=2}^{d} \Delta(x_i^n,h) \bigr\}\,;\quad I^{n-} :=
\mathbb{I}_{\,U^{n}<a^{-}(X^n,Y^n)\,}\,,
\end{equation}
and set
\begin{equation*}
\xi^{n} = I^{n-}(\,\mathcal{P}_q\,\psith(q^n,p^n) - q^{n}\, )^2 \
.
\end{equation*}
Using the Lipschitz continuity of $u\mapsto \mathbb{I}_{\,U\,\le
\,1\wedge\, e^{u}\,}$ and the Cauchy-Schwartz inequality we get:
\begin{equation*}
\E\,| (q^{n+1}-q^{n})^2 -\xi^{n}| \le |\Delta(x_{1},h)|_{L_2}\,|(
\mathcal{P}_q \psith(q^n,p^n) - q^{n} )^2 |_{L_2}
\end{equation*}
Now, Conditions \ref{cond:delta} and \ref{cond:beta} imply that
\begin{equation*}
|\Delta(x_{1},h)|_{L_2} = \mathcal{O}(h^2)\ .
\end{equation*}
Also, from Condition \ref{cond:E} and the stated hypothesis on the density
$\exp(-V)$, $q^{n}$ and $\mathcal{P}_q \psith(q^n,p^n)$ have bounded
fourth moments uniformly in $h$, so:
\begin{equation*}
|( \mathcal{P}_q \psith(q^n,p^n) - q^{n} )^2 |_{L_2} \le C \ ,
\end{equation*}
for some constant $C>0$. The last two statements imply that:
\begin{equation}
\label{eq:eq} \E\,| (q^{n+1}-q^{n})^2 -\xi^{n}| = \mathcal{O}(h^2)
\ .
\end{equation}
Exploiting the independence between $I^{n-}$ and the first
particle:
\begin{multline*}
\E\,[\,\xi_n\,] = \E\,[\,a^{-}(X,Y)\,]\times \E\,[\,( \mathcal{P}_q \psith(q^n,p^n) - q^{n} )^2\,] \longrightarrow \\
a(l)\cdot \E\,[\,(\mathcal{P}_q \psit(q^n,p^n)-q^{n})^2\,] \ ,
\end{multline*}
where, for the first factor we used its limit from Theorem
\ref{thm:lim}; for the second factor the limit is a consequence by
Condition 3 and the dominated convergence theorem. Equation
(\ref{eq:eq}) completes the proof.
\end{proof}

\begin{proof}[Proof of Proposition \ref{thm:lim1}]
Fix some $q_1^{n}\in\R^m$. We define $a^{-}(X^n,Y^n)$ and $I^{n-}$
as in (\ref{eq:aaa}). For simplicity, we will write just $q^{n}$,
$q^{n+1}$, $\mathsf{q}^{n+1}$ and $p^{n}$ instead of $q_{1}^n$,
$q_{1}^{n+1}$, $\mathsf{q}_1^{n+1}$ and $p_{1}^n$ respectively.

We set
\begin{equation*}
g^{n+1}= I^{n-}\cdot \mathcal{P}_q
\psit(q^{n},p^{n})+\bigl(1-I^{n-})\,q^{n} \ .
\end{equation*}
Adding and subtracting
$I^{n}\cdot\mathcal{P}_{q}(\psit(q^{n},p^{n}))$  yields:
\begin{align}
|q^{n+1}-g^{n+1}| \le |\mathcal{P}_{q}(\psith(q^{n},p^{n}))&-\mathcal{P}_{q}(\psit(q^{n},p^{n}))| \nonumber \\
&+ |I^{n-}-I^{n}|\,\bigl(|\mathcal{P}_{q}(\psit(q^{n},p^{n}))| +
|q^{n}|\bigr) \ . \label{eq:diff}
\end{align}
Using the Lipschitz continuity (with constant 1) of $u\mapsto
\mathbb{I}_{\,U\,\le \,1\wedge\, \exp(u)\,}$:
\begin{equation}
\label{eq:diff1} |I^{n-}-I^{n}| \le |\Delta(x_{1},h)|\ .
\end{equation}
Now, Condition \ref{cond:E} implies that the first term on the right-hand
side of (\ref{eq:diff}) vanishes w.p.\@1 and Condition
\ref{cond:delta} implies (via (\ref{eq:diff1})) that also the
second term  vanishes w.p.\@1. Therefore, as $d\rightarrow
\infty$:
\begin{equation*}
q^{n+1}-g^{n+1} \rightarrow 0,\,\,\,\textrm{a.s.}\ .
\end{equation*}
Theorem \ref{thm:lim} immediately implies that $I^{n-}
\stackrel{\mathcal{L}}{\longrightarrow} \mathsf{I}^{n}$, thus:
\begin{equation*}
g^{n+1} \stackrel{\mathcal{L}}{\longrightarrow} \mathsf{q}^{n+1} \
.
\end{equation*}
From these two limits,  we have
$q^{n+1}\stackrel{\mathcal{L}}{\longrightarrow} \mathsf{q}^{n+1}$,
and this completes the proof.
\end{proof}

\begin{proof}[Proof of Theorem \ref{thm:sjd}]
To simplify the notation we again drop the subscript 1.
Conditionally on the trajectory $q^0, q^{1},\ldots$ we get:
\begin{equation*}
(q(t+\delta)-q(t))^2 = \left\{
\begin{array}{ll}
0\, , & \textrm{w.p.}\,\,\,1-\lambda_{d}\delta + \mathcal{O}((\lambda_d\delta)^2)\ ,\\
(q^{N(t)+1}-q^{N(t)})^2\, ,& \textrm{w.p.}\,\,\,\lambda_{d}\delta + \mathcal{O}((\lambda_d\delta)^2)\ ,\\
(q^{N(t)+1+j}-q^{N(t)})^2\,,\,\,j\ge 1\, , &
\textrm{w.p.}\,\,\,\mathcal{O}((\lambda_d\delta)^{j+1})\ .
\end{array}
\right.
\end{equation*}
Therefore,
\begin{align}
\SJ_{d} = \E\,[\,&(q^{N(t)+1}-q^{N(t)})^2\,]\, (\lambda_{d}\delta
+ \mathcal{O}((\lambda_d\delta)^2)) \nonumber
\\ &+\sum_{j\ge 1}\E\,[\,(q^{N(t)+1+j}-q^{N(t)})^2\,]\,\mathcal{O}((\lambda_d\delta)^{j+1}) \ .
\label{eq:ss}
\end{align}
Note now that:
\begin{align*}
\E\,[\,(q^{N(t)+1+j}-q^{N(t)})^2\,] &\le \Bigl(\,
\sum_{k=1}^{j+1}| q^{N(t)+k} -q^{N(t)+k-1}|_{L_2}\, \Bigr)^2 \\
&=(j+1)^2\,\E\,[\,(q^{n+1}-q^{n})^2\,] \ ,
\end{align*}
since we have assumed stationarity.
From (\ref{eq:cost}):
\begin{equation*}
\lambda_d = d^{-5/4}\,\frac{l}{T\,C_{LF}} +
\mathcal{O}(d^{-6/4})\ .
\end{equation*}
and, from Proposition \ref{th:desplaz}, $\E\,[\,(q^{n+1}-q^{n})^2\,] =
\mathcal{O}(1)$. Therefore,
\begin{equation*}
d^{5/4}\times \sum_{j\ge
1}\E\,[\,(q^{N(t)+1+j}-q^{N(t)})^2\,]\,\mathcal{O}((\lambda_d\delta)^{j+1})
\end{equation*}
is of the same order in $d$ as
\begin{equation*}
\lambda_d^2\cdot d^{5/4}\times \sum_{j\ge 1}(j+1)^2\;
\mathcal{O}(\lambda_d^{j-1}) \ ,
\end{equation*}
thus:
\begin{equation*}
d^{5/4}\times \sum_{j\ge
1}\E\,[\,(q^{N(t)+1+j}-q^{N(t)})^2\,]\,\mathcal{O}((\lambda_d\delta)^{j+1})
= \mathcal{O}(\lambda_d) \ .
\end{equation*}
Using this result, and continuing from (\ref{eq:ss}), 
Proposition \ref{th:desplaz}  provides the required statement.
\end{proof}

\section{Conclusions}
\label{sec:conclusion} The HMC methodology provides
a promising framework for the study of a number of
sampling problems, especially in high dimensions.
There are a number of directions
in which the research direction taken in this
paper could be developed further. We list some of them.

\begin{itemize}

\item The overall optimization involves tuning  \emph{three} free
parameters $(l,T,M)$, and since M is a matrix, the number of
parameters to be optimized over, is even more in general.
In this paper, we have fixed $M$  and
$T$, and  focussed on optimizing the HMC algorithm
over choice of step-size $h$. The natural next step would be 
to study the algorithm for various  choices of the mass matrix ${M}$ 
and the integration time $T$.

\item We have concentrated on explicit integration by
the leapfrog method. For measures which have density
with respect to a Gaussian measure (in the limit $d \to
\infty$) it may be of interest to use semi-implicit
integrators. This idea has been developed for
the MALA algorithm (see \cite{Besk:Stua:09} and the
references therein) and could also be developed for
HMC methods. It has the potential of leading to
methods which explore state space in $\mathcal{O}(1)$
steps. 

\item The issue of irreducibility for the transition kernel
of HMC is subtle, and requires further investigation,
as certain exceptional cases can lead to nonergodic
behaviour
(see \cite{Canc:Stol:07,Schu:98} and the
references therein).

\item There is evidence that the limiting properties of MALA for 
high-dimensional target densities 
do not appear to depend critically on the tail behaviour of the target 
(see \cite{robe:98}). However in the present paper for HMC, we have 
considered densities that are no lighter than 
Gaussian at inÞnity.  It would thus be interesting to extend the work to light-tailed 
densities. This  links naturally to the question of using variable step size integration for 
HMC since light tailed densities will lead to superlinear vector fields at
 infinity in \eqref{eq:ham1}.

\item There is interesting recent computational 
work \cite{giro:09} concerning exploration of state space 
by means of nonseparable Hamiltonian dynamics; this work 
opens up several theoretical research directions.

\item We have shown how to scale
the HMC method to obtain $\mathcal{O}(1)$ acceptance probabilities
as the dimension of the target product measure grows. We
have also shown how to minimize a reasonable measure of
computational cost, defined as the work needed to make
an $\mathcal{O}(1)$ move in state space. However, in contrast
to similar work for RWM and MALA (\cite{robe:97, robe:98}) we have not
completely identified the limiting Markov process
which arises in the infinite dimensional limit. This
remains an interesting and technically demanding challenge.

\end{itemize}

\section*{Acknowledgements}
We thank Sebastian Reich for drawing our attention to the 
paper \cite{gupt:90} which
sparked our initial interest in the scaling issue for HMC. 
Further thanks also to Gabriel Stoltz for 
stimulating discussions.
Part of this paper was written when
NSP was visiting JMS at the University of Valladolid and we thank this institution
for its warm hospitality.
\bibliographystyle{amsplain}
\bibliography{references}

\end{document}